\documentclass[12pt,a4paper]{article}
\usepackage{amsmath,amscd,amssymb,amsthm,amstext,mathrsfs,dcolumn,
booktabs,enumerate}
\usepackage[shortlabels]{enumitem}
\usepackage[english]{babel}
\usepackage{geometry}
\usepackage{ifpdf}
\ifpdf
\usepackage[all,line,cmtip]{xy}
\else
\usepackage[all,line,cmtip,dvips]{xy}
\fi

\input{xypic}

\theoremstyle{plain}%default
   \newtheorem{theorem}{Theorem}[section]
   \newtheorem{proposition}[theorem]{Proposition}
   \newtheorem{lemma}[theorem]{Lemma}
   \newtheorem{corollary}[theorem]{Corollary}
   \theoremstyle{definition}
   
   \newtheorem{definition}[theorem]{Definition}
   
   \theoremstyle{remark}
   
   \newtheorem{remark}[theorem]{Remark}
   
   \newtheorem{condition}[theorem]{Condition}  
 \newtheorem{addendum}[theorem]{Addendum}

\newcommand{\RR}{{\mathbb R}}
\newcommand{\QQ}{{\mathbb Q}}

\newcommand{\Sd}{{\mathrm {Sd}}}
\newcommand{\inter}{{\mathrm {int}}}
\newcommand{\Aset}[3]{{A_#1(\underline{#2},#3)}}
\newcommand{\Jset}[3]{{J_#1(\underline{#2},#3)}}
\newcommand{\norm}[1]{\vert #1\vert}
\newcommand{\Burns}[3]{???}
\usepackage{mathtools}
\DeclarePairedDelimiter{\Realization}{\lvert}{\rvert}
\DeclarePairedDelimiter{\Norm}{\lVert}{\rVert}
\title{The cobordism category and Waldhausen's $K$-theory.}
\author{Marcel B\"o{}kstedt \\
Department of Mathematical Sciences\\
Aarhus University\\
Ny Munkegade\\
8000 Aarhus C\\
Denmark\\
\texttt{@email:marcel@imf.au.dk}
\and
Ib Madsen\\
Department of Mathematical Sciences\\
University of Copenhagen\\
Universitetsparken 5\\
2100 K\o{}benhavn \O{}\\
Denmark\\
\texttt{@email:imadsen@math.ku.dk}
}

\begin{document}

\maketitle

\tableofcontents

\setcounter{section}{-1}
\section{Introduction}

This paper examines the category 
${\cal C}^k_{d,n}$ whose morphisms are $d$-dimensional smooth
manifolds that are properly embedded in $I^k\times \RR^{d+n-k}$,
where $I^{ k}$ is a $k$-dimensional cube. There are $k$ directions to
compose $k$-dimensional cubes, so ${\cal C}^k_{d,n}$ is a (strict)
$k$-tuple category. The geometric realization of the $k$-dimensional
multi-nerve $N_{\bullet}{\cal C}^k_{d,n}$ is the classifying space
$B{\cal C}^k_{d,n}$. Its homotopy type is determined by theorem
\ref{th:homotopyType} below to be 
\begin{equation}
\label{eq:introHE}
B{\cal C}^k_{d,n}\simeq \Omega^{d+n-k}\mathrm{Th}(U^\perp_ {d,n})
\end{equation}
where $U_{{d,n}}^{\perp}$ is th $n$-dimensional canonical bundle over the
Grassmannian $G(d,n)$ of $d$-planes in $\RR^{n+d}$ and where 
$\mathrm{Th}$ denotes the Thom space, that is the one point
compactification  of $U_{d,n}^{\perp}$. 

For $k=1$ and $n=\infty$, the structure of 
$B{\cal C}^k_{d,n} $ was determined in \cite{GMTW}, using the sheaf
techniques of \cite{MadsenWeiss}. We note from (\ref{eq:introHE}) that
\[
\Omega B{\cal C}^k_{d,n}\simeq B{\cal C}^{k-1}_{d,n},\quad k\leq d+n.
\]
For $n\to  \infty$ we get a geometric interpretation of the $\Omega$-spectrum $MTO(d)$ of \cite{GMTW}, namely
\begin{equation}
\label{eq:introMTO}
MTO(d)\cong \big\{B{\cal C}^k_{d,\infty}\big\}^\infty_{k=1}.
\end{equation}
At the end of the paper in \S \ref{sec:Atheory}
we use (\ref{eq:introMTO}) to construct an infinite loop map
\begin{equation}
\label{eq:infLoop}
\Omega B{\cal C}^1_{d,n}\to A(BO(d))
\end{equation}
We believe that the map factors through
$\Omega^\infty\Sigma^\infty(BO(d)_+)$
and that the composite
\[
B\mathrm{Diff}(M^d)\to \Omega B{\cal C}^1_{d,n}\to A(BO(d))
\]
is homotopic to the map considered in \cite{DWW}.
 
Our method of proof of (\ref{eq:introHE}) is rather different from 
\cite{GMTW}. We begin with the abstract transversality theorem of
\S \ref{sec:critical-pairs-cuts}: Given a metric space $X$ and a closed subspace
$Z\subset X \times \RR^k$ with the property that $Z\cap(\{x\}\times \RR^k)$
has measure 0, we introduce a simplicial space 
$\vert K_\bullet(X,Z)\vert$ with a map to $X$ and show in theorem \ref{th:cuts} that 
\begin{equation}
\label{eq:introCUTS}
\vert K_\bullet(X,Z)\vert\xrightarrow{\sim} 
\end{equation}
is a homotopy equivalence.

Let $\Psi_d(\RR^{d+n})$ be the space of properly embedded $d$-dimensional smooth submanifolds of $\RR^{d+n}$ equipped with a topology where, roughly speaking, two manifolds are close if one is contained in a tubular neighbourhood of the other. For the space $X$ in (\ref{eq:introCUTS}) we take the subspace
$D^k_{d,n}\subset \Psi_d^{d+n}$ where $M\in D^k_ {d,n}$ if the projection
on the first $k$ coordinates is a proper map from $M$ to $\RR^k$. The space
$Z=Z^k_{d,n}\subset D^k_{d,n}\times \RR^k$ consists of pairs
$(M,\underline a)$ where $M$ fails to be transversal to the ``corners''
determined by $\underline a\in \RR^k$.

It turns out that $K_\bullet(D^k_{d,n},Z^k_{d,n})$ is homotopy
equivalent to the multi-nerve $N_\bullet({\cal C}^k_{d,n})$, so by (\ref{eq:introCUTS}) 
$B({\cal C}^k_{d,n})\simeq D^k_{d,n}$. An application of Gromov's general h-principle yields that $D^k_{d,n}$ is homotopy equivalent to the right hand side of (\ref{eq:introHE}).

The proof of (\ref{eq:introCUTS}) needs that $X=D^k_{d,n}$ or
$\Psi_d(\RR^{d+n})$ be metrizable. This is proved in the technical 
\S \ref{sec:topology}. It is finally in order to remark that \S
\ref{sec:simplicial-spaces} contains results about simplicial spaces, needed in the proof of (\ref{eq:introCUTS}) which may be of general interest.
 
We would like to acknowledge the inspiration from S. Galatius'
manuscript \cite{Galatius}.
\section{Abstract transversality}
\label{sec:critical-pairs-cuts}

A critical pair is a space $X$ and  a closed subset $Z\subset X \times \RR^k$.
We call $Z$ the critical datum, and use the notation $Z(x)=\{v\in \RR^k\mid (x,v)\in Z\}$.

Let $\RR^{k}$ be partially ordered by $\underline{a}_1 \leq \underline{a}_2$ if the inequality
is valid for each coordinate, that is if 
$a^{i}_1\leq a^{i}_2$ for $1\leq i \leq k$. For a totally  ordered
pair of vectors $\underline{a}_1 \leq \underline{a}_2\in \RR^k$ we consider the
set of $2^k$ vectors obtained by mixing the coordinates, i. e. the set
of $2^k$ vectors in the cube with southwest vertex $\underline{a}_{1}$ and
northeast vertex $\underline{a}_{2}$:
\[
V(\underline{a}_{1},\underline{a}_{2})=\{(a^1_{f(1)},a^2_{f(2)},\dots, a^k_{f(k)})\}\qquad f\colon \left\{1,\dots,k\right\} \to \left\{1,2\right\}.
\]
\begin{definition}
\label{def:compatible}
 A pair $\underline{a}_1\leq \underline{a}_2$ is compatible with $Z(x)\subset\RR^k$ 
if none of the vectors in 
$V(\underline{a}_{1},\underline{a}_{2})$
 is contained in $Z(x)$.  
\end{definition}
 
In particular, if $\underline{a}_{1} \leq \underline{a}_{2}$ are compatible, then $\underline{a}_{i}\not\in Z(x)$.

\begin{definition}
\label{def:cutSets}
The simplicial space of cut sets $K_\bullet(X,Z)$ is
given by
\[
K_q(X,Z) = \{x\in X, 
\underline{a}_{0} \leq \underline{a}_{1} \leq \dots \leq \underline{a}_{q}\text{; all pairs
  $\underline{a}_{i}\leq \underline{a}_{j}$ compatible}\},
\]  
topologized as subspace of $X\times \RR^{k(q+1)}$.

The simplicial space of discrete cut sets denoted by
$K_q^{\delta}(X,Z)$
is the same underlying set but topologized as a subset of $X\times (\RR^{k(q+1)\delta})$.
That is, we do not change the topology in the  $X$
direction, but discretize the topology in the $\RR^{k}$
direction.
\end{definition}

Let $N_\bullet(\RR^k)$ be the nerve of the partially ordered set
$\RR^{k}$. Then $K_q(X,Z)$ is a subspace of the product $X \times
N_{q}(\RR^{k})$, and we define the simplicial structure maps of
$K_\bullet(X,Z)$ so that this is an inclusion of simplicial spaces. 

The simplicial space $K_\bullet(X,Z)$ is the diagonal of a
$k$-dimensional simplicial space, since the $k$-dimensional boxes
of $K_1(X,Z)$ can be composed in $k$ different directions. For
$k=2$, $\omega=(x,\underline{a}_{0} \leq \underline{a}_{1}) \in K_1(X,Z)$ can be pictured as
the square in $\RR^{2}$ whose vertices $V(\underline{a}_0,\underline{a}_{1})$ all lie outside
   $Z(x)$. These squares can be composed horizontally and
    vertically, and we let $K_{p,q}(X,Z)$ denote the grid of $p \times
    q$ squares all of whose vertices are outside $Z(x)$. Removing
    vertical and horizontal edges defines the bi-simplicial
structure on $K_{\bullet,\bullet}(X,Z)$. Its diagonal simplicial
space is $K_{\bullet}(X,Z)$. The situation is similar for $k>2$.

\begin{remark}
  One can interpret $K_{\bullet}(X;Z)$ as the nerve of a 
strict $k$-tuple category in the sense of \cite{Leinster}, defined
inductively as follows: A strict 0-category is a set, and a strict  
$k$-tuple category is
 a category object in the category of strict $k-1$-tuple categories.
A strict 2-tuple category is a pair of small categories ${\cal C}_{0},{\cal
  C}_{1}$
together with functors
\[
\xymatrix{
 {\cal C}_{1}\ar@<6pt>[r]^{d_{1}}\ar@<-6pt>[r]_{d_{0}}&{\cal C}_{0}\ar[l],&
\quad {\cal C}_{1}\times_{{\cal C}_{0}}{\cal
  C}_{1}\xrightarrow{\circ}{\cal C}_{1},
}
\]
In particular, the objects of ${\cal C}_{1}$ are the morphism of
a category ${\cal C}^{\prime}$, whose objects are the objects of ${\cal C}_{0}$.

An element $\omega\in N_{1}{\cal C}_{1}=\mathop{mor} {\cal C}$ gives rise to a square
\[
\xymatrix@C=1.5cm{
\underline{a}_{1,1}\ar[d]^{f_{1,1}}\ar[r]^{g_{1,1}}&\underline{a}_{0,1}\ar[d]^{f_{0,1}}\\
\underline{a}_{1,0}\ar[r]^{g_{1,0}}&\underline{a}_{0,0}\ ,\\
}
\]
where $g_{1,1}=d_{0}\omega$, $g_{1,0}=d_{1}\omega$ and $\underline{a}_{i,j}$
is the source and target of the morphisms $g_{1,1}$ and $g_{1,0}$
in ${\cal C}_{0}$. The vertical arrows are the objects in
${\cal C}_{1}$ which are the source and target of $\omega$. 
They are considered as morphisms in ${\cal C}^{\prime}$. 

These diagrams can be composed horizontally and vertically
to defines bi-simplicial set (or space, in case one deals with
topological categories) 
$N_{\bullet,\bullet}({\cal C}_{1},{\cal C}_{0})$.

In the case of a critical pair $Z\subset X\times \RR^{2}$, the
associated 2-tuple category is defined as follows:
\begin{align*}
K_{0}(X;Z)&=N_{0}{\cal C}_{0}= X\times \RR^{2}\setminus Z,\\
N_{1}{\cal C}_{0}&= \{(x,\underline{a}_{1},\underline{a}_{2})\in X \times \RR^{2}\times
\RR^{2}\mid
\underline{a}_{1}\leq \underline{a}_{2},\underline{a}_{1}^{2}=\underline{a}_{2}^{2},
V(\underline{a}_{1},\underline{a}_{2})\cap Z(x)=\emptyset\},
\\
N_{0}{\cal C}_{1}&= \{(x,\underline{a}_{1},\underline{a}_{2})\in X \times \RR^{2}\times
\RR^{2}\mid
\underline{a}_{1}\leq \underline{a}_{2},\underline{a}_{1}^{1}=\underline{a}_{2}^{1},
V(\underline{a}_{1},\underline{a}_{2})\cap Z(x)=\emptyset\},
\\
K_{1}(X;Z)&=N_{1}{\cal C}_{1}=\{(x,\underline{a}_{1},\underline{a}_{2})\in X\times \RR^{2}\times \RR^{2}\mid
\underline{a}_{1}\leq \underline{a}_{2}, V(\underline{a}_{1},\underline{a}_{2})\cap Z(x)=\emptyset\}.
\end{align*}
Then $K_{\bullet,\bullet}(X;Z)=N_{\bullet,\bullet}({\cal C}_{1},{\cal
  C}_{0})$.
The situation for $k\geq 2$ is similar, since the nerve of a
simplicial object in the category of $(k-1)$-dimensional simplicial
sets
is a $k$-dimensional simplicial set.
\end{remark}

Our first result is that the topology on the Euclidean factor does not
matter much.

\begin{theorem}
\label{th:notopology}
Assume that $X$ is a metrizable space.
The map $\lambda_{Z} : K_\bullet^{\delta}(X;Z) \to K_\bullet(X;Z)$
is a weak homotopy equivalence.
\begin{proof}
We will consider certain subspaces of 
$K_\bullet(X;Z)$. For a subset  $W \subset \RR^{k}$,  
let $K_\bullet(X;Z)(W)$ consist of all simplices of $K_\bullet(X;Z)$ whose
vertices are contained in $W$. Given a simplex 
$(x,\underline{a}_0,\dots,\underline{a}_r)\in K_r(X;Z)$
with $\underline{a}_0 < \dots <\underline{a}_r$, we can find an open
neighbourhood $U\subset X$ of $x$ and disjoint axis-parallel
cubes $W_i$ containing $\underline{a}_i$ with the following properties.

\begin{enumerate}[(i)]
\item The set $\{W_i\}$ is totally ordered in the sense that for
$\underline{w}_i\in W_i$, $\underline{w}_0<\dots <\underline{w}_r$.
\item For $y\in U$ and $\underline{w}_0<\dots <\underline{w}_r$ as in (i),
$(y,\underline{a}_0,\dots,\underline{a}_r) $ determines a
$k$-simplex of $K_\bullet(X;Z)$, i.e.
$K_{\bullet}(U;Z\cap (U\times \RR^{k}))(W)=K_{\bullet}(U;\emptyset)(W)$,
where $W$ is the disjoint union of the $W_i$.
\end{enumerate} 
This follows because the condition that
$(y,\underline{a}_0,\dots,\underline{a}_r)$ belongs to
$K_{\bullet}(U;\emptyset)(W)$ is an open condition.

We consider $W=\coprod{W_i}\subset \RR^k$ and $U\subset X$ of
the above type (satisfying (i) and (ii) above, 
and note that if $(U,W)$ and
$(U,W^\prime)$ are two such pairs, then so is
$(U,W\cap W^\prime)$. 

{\em Claim 1.} If $(W,U)$ is a pair satisfying the two conditions, 
the maps
\[
\lambda_{Z} : K_\bullet^{\delta}(U;Z\cap (U\times \RR^{k}))(W) 
\to K_\bullet(U;Z\cap(U\times \RR^{k}))(W)
\to U
\]
are both homotopy equivalences. 
Because of the second condition, 
$K_\bullet(U;Z\cap (U\times \RR^{k}))(W)\cong
U \times K_{\bullet}(p,\emptyset)(W)$, 
where $p$ is a one point space, so it is enough to show that 
both $K_{\bullet}(p,\emptyset)(W)$ and
$K_{\bullet}^{\delta}(p,\emptyset)(W)$ are contractible.

The usual cofinality argument easily shows that 
$K_{\bullet}^{\delta}(p,\emptyset)(W)$ 
is contractible. Indeed,
it's enough to show that if $A$ is a compact space, then any map
$f:A\to \Realization{K_{\bullet}^{\delta}(p,\emptyset)(W)}$ 
is homotopic to a constant map.
But the image of a compact set will be contained in a finite simplicial subset
$X_\bullet\subset K_{\bullet}^{\delta}(p,\emptyset)(W)$. 
The vertices of $X_\bullet$ are given by finitely many points of $W$. 
To a finite set of points in $W$ there is a point $\underline b\in W$ strictly 
greater than each one of them. The
image of $f$ is then contained in the cone of $X_\bullet$ with vertex at $\underline b$, so that
$f$ is homotopic to a constant map.

Now the claim is reduced to proving that 
$K_{\bullet}(p,\emptyset)(W)=N_\bullet(W,\leq)$ is also contractible.
This simplicial space is the nerve of the topological partially 
ordered set $W$.
Suppose that $W=\cup_{1\leq i\leq c} W_{i}$ where $W_{i}$ are the
components of $W$, in the given total order. Let $\underline a<\underline b\in \RR^{k}$ be
given such that the largest component of $W$ is
\[
W_{c}=\{\underline x\in \RR^{k}\mid \underline a<\underline x<\underline b\}.
\]
and let 
\[
W_{c}^{\prime}=\{\underline x\in \RR^{k}\mid \underline a < \underline x < (\underline a+\underline b)/2\}\subset W_{c}.
\]
We put $W^{\prime}=(\cup_{1\leq i \leq c-1}W_{i})\cup W_{c}^{\prime}$.
Let $i\colon W^{\prime}\to W$ be the inclusion, and let
$h:W\to W^{\prime}$ the map which is the identity on $W_{i}$ for
$i\leq c-1$, and $h(\underline x)=(\underline a+\underline x)/2$ for $\underline x\in W_{c}$. 
Since $i\circ h(\underline x)\leq \underline x$, there is a
natural transformation from the functor $i\circ h$ of the category
$(W,\leq)$ to the
identity, so that the identity on $K_{\bullet}(p,\emptyset)(W)$
factors up to homotopy over the inclusion
$i:K_{\bullet}(p,\emptyset)(W^{\prime})\to
K_{\bullet}(p,\emptyset)(W)$. 
Let $w\in W$ be larger than every element in $W^{\prime}$. The nerve of
$\{W^\prime \cup \{w\}\}$ is contractible, since it has a largest
element. If follows that the composite 
$i:W^\prime \subset \{W^\prime \cup \{w\}\} \subset W$ induces a map on
nerves which is homotopic to a constant map, This proves 
that $K_{\bullet}(p,\emptyset)(W)$ is contractible,
and finishes the proof of claim 1.

Note that it follows trivially from this computation the natural map 
$K_{\bullet}^{\delta}(p,\emptyset)(W)\to K_{\bullet}(p,\emptyset)(W)$ 
is a homotopy
equivalence. 

{\em Claim 2.} For any finite set of such $\{W^j\}_{1\leq j\leq n}$ and 
$U^j\subset X$, $j=1,\dots,n$ which satisfy (i) and (ii), 
the natural map
\[\bigcup_{j}\Realization{K_{\bullet}^{\delta}(U^j;Z\cap(U^j\times \RR^{k}))(W^j)}
\to\bigcup_{j}\Realization{K_{\bullet}(U^j;Z\cap(U^j\times \RR^{k}))(W^j)}
\]
 is a homotopy equivalence.
The proof is by induction on $n$. The induction start is claim 1
above. Let us put
$A_{\bullet}^{j}=K_{\bullet}(U^{j};Z\cap(U^{j}\times
  \RR^{k}))(W^{j})$ and 
$A_{\bullet}^{j,\delta}=K_{\bullet}^\delta(U^{j};Z\cap(U^{j}\times
  \RR^{k}))(W^{j})$
Consider the diagram
\[
\begin{CD}
\bigcup_{1\leq j\leq n-1}\Realization{A_{\bullet}^{j,\delta}}
@<<<
(\bigcup_{1\leq j\leq n-1}\Realization{A_{\bullet}^{j,\delta}})\cap
\Realization{A_{\bullet}^{n,\delta}}
@>>>
\Realization{A_{\bullet}^{j,\delta}}\\
@VVV @VVV @VVV \\
\bigcup_{1\leq j\leq n-1}\Realization{A_{\bullet}^{j}}
@<<<
(\bigcup_{1\leq j\leq n-1}\Realization{A_{\bullet}^{j}})\cap
\Realization{A_{\bullet}^{n}}
@>>>
\Realization{A_{\bullet}^{n}}
\end{CD}
\]
By induction the left and the middle vertical maps are homotopy equivalences,
and claim 1 says the right vertical map is a homotopy equivalence,
so the diagram induces a homotopy equivalence from the homotopy pushout
of the upper row to the homotopy pushout of the lower row.
We need to show that the homotopy pushout of the rows are
homotopy equivalent to the degreewise pushouts.

Degreewise, the pushout
\[
\begin{CD} 
(\bigcup_{1\leq j\leq n-1}A_{p}^{j})\cap
A_{p}^{n} @>>>
\bigcup_{1\leq j\leq n-1}A_{p}^{n}\\
@VVV @VVV \\
A_{p}^{n}@>>>
\bigcup_{1\leq j\leq n}A_{p}^{n}
\end{CD}
\]
is a homotopy pushout, because it is the union of two open sets
in a normal space (this uses the condition that $X$ is metrizable). 
But the realization of a degreewise homotopy pushout
diagram of simplicial spaces is a homotopy pushout, because
the realization of a degreewise mapping cylinder is
homeomorphic to the mapping cylinder of the realization.
This concludes the induction step, and finishes the proof of 
claim 2.

{\em Claim 3.} Let $p\in \Realization{ K_\bullet(X;Z)}$. There is a set $W^{p}\subset \RR^{k}$ satisfying 
conditions (i) and (ii) above, and an open set $U^{p}\subset X$, such that
$p$ is in the image of the natural map 
\[
\Realization{K_{\bullet}(U^{p},Z\cap (U^{p}\times \RR^{k}))(W^p)})\to
\Realization{ K_\bullet(X;Z}.
\]
The point $p$ is in the
image of a characteristic map of a simplex determined by
$x\in X$ and a totally ordered set of 
vectors $\underline{a}_{0} < \underline{a}_{1} < \dots < \underline{a}_{r}$. The point $x$ and the
vectors $\underline{a}_{i}$ satisfy a number
of conditions, determined by the closed set $Z\subset X\times
\RR^{k}$. 
These conditions are open conditions. This means that there is an
open set $U^{p}\subset X$ and a set of
open cubes $C_{i}(a)\ni \underline{a}_{i}$, such that $C_{0}<C_{1}<\dots <C_{k}$,
and so that if we put $W^{p}=\cup_{i}C_{i}$ $u\in U$, then
the simplex in $K_{r}(X;Z)$ given by $(u,\underline b_{0}\leq \underline b_{1}\leq \dots \leq
\underline b_{r})$
is contained in $K_r(U;Z\cap(U\times \RR^{k}))$ 
if all $\underline b_j\in W^{p}$ and  $u\in U^{p}$. This proves claim
3.In conclusion, we have a covering of $\Realization{K(X;Z)}$ by 
realizations of simplicial subsets
$K_{\bullet}(U,Z\cap(U\times \RR^{k}))$. According to
claim 2, these sets have the property that the 
map from the corresponding discrete version is a homotopy equivalence.
If these subsets were open, we would
be done. In general they are not open, but they are degreewise
open in the simplicial space $K_{\bullet}(X;Z)$. Theorem~\ref{th:opencover} below completes the proof.
\end{proof}
\end{theorem}

We can now deal with the simplest special case.

\begin{lemma}
\label{le:pointspace}
Let $p$ be a one point space, $Z\subset \RR^k$ a closed set of measure zero.
Then $\Realization{K(p;Z)}$ is weakly contractible.
\begin{proof}
By theorem~\ref{th:notopology}, 
it is sufficient to prove that 
$\Realization{K_{\bullet}^{\delta}(p,Z)}$ 
is weakly contractible. It's enough to show that the 
inclusion of any finite subcomplex 
of $\Realization{K_{\bullet}^{\delta}(p,Z)}$ is homotopic
to a constant map. Any finite subcomplex will involve finitely many simplices,
which are defined using a finite set of points $\underline a\in \RR^{k}\setminus
Z$. 

Assume that $A=\{\underline{a}_i\}_{1\leq i\leq n}$ is a
finite subset of $\RR^{k}\setminus Z$, and 
$i_A\colon G(A)\subset K_\bullet^{\delta}(p,Z)$ the simplicial subset 
consisting of all simplices of $K_\bullet^{\delta}(p,Z)$ with vertices in $A$.
It suffices to show that 
the map $i_A$ is homotopic to a constant map. 
We show inductively that for any $0\leq m\leq n$ we can find a $\underline b$ such
that $\underline{a}_{i}\leq \underline b$ for $1\leq i\leq n$, and such that $\underline{a}_{i}$ and $\underline b$
are compatible for $i\leq m$.
It follows that $G(A\cup \{\underline b\})$ 
is contractible, since it is a cone on the vertex $\underline b$. For $m=0$ pick any point $\underline b\in \RR^k$ larger than all the $\underline{a}_{i}$.  

To do the induction step, assume that we can find a 
$\underline b^{\prime}$ such that $\underline{a}_{i} \leq
\underline b^{\prime}$ for all $i$, and such that $\underline{a}_{i}$ and $b^{\prime}$ are compatible 
for $1\leq i\leq m-1$. Since $Z$ is closed, there is
an open cube centered on $\underline b^{\prime}$, such that any point in this cube is also
greater than all $\underline{a}_{i}$ and compatible with $\underline{a}_{i}$, $1\leq i\leq m-1$. 
Since $Z$ has measure 0, there is at least one point 
$\underline b$ in this cube that satisfies
that $\underline{a}_{m}$ and $\underline b$ are compatible. 

\end{proof}
\end{lemma}

\begin{lemma}
Let $X$ be a metrizable space.
 The simplicial space $K_\bullet =K_\bullet(X,Z)$ is a good simplicial space in the sense
of \cite{SegalCategories}, so that the fat realization $\lVert
K_{\bullet}\rVert$ is homotopy equivalent to the standard realization
$\lvert
K_{\bullet}\rvert$. 
\begin{proof}
It suffices to show that each degeneracy map
$K_p \to K_{p+1}$ is a cofibration. Let $(x,{\underline v}_0,\dots {\underline v}_p)$ where $x\in X$
and ${\underline v}_i\in \RR^k$.  The  degeneracy map $s_i$ iterates the vector ${\underline v}_i$. We can identify $K_p$ with
the subset $A \subset K_{p+1}$, defined by the
equation that $\underline v_i=\underline v_{i+1}$. In particular, the image of the
degeneracy map is closed. 
Since $X$ is metrizable, $K_{p+1}$ is normal, and it suffices to prove
that
$A\subset K_{p+1}$ is a deformation retract of a neiborhood, cf
\cite{Puppe},
Satz 1. 
   
Let $h(x,\underline v)$ be the distance from $\underline v$ to $Z(x)$. 
We claim that this function is 
upper semi-continuous, in the sense that for any 
$c\in \RR$ the set $\{(x,\underline v)\in X\times\RR^k \mid h(x,\underline v) > c\}$ is open
in $X \times \RR^k$. To see this, consider a point $(x_0,{\underline v}_0)$ such that
$h(x_0,{\underline v}_0) >c$. Choose a number $q$, such that $h(x_0,{\underline v}_0)> q  > c$. We can cover the closed disc $\{x_0\} \times D({\underline v}_0,q)$ with
open sets in $X \times \RR^k$, disjoint from $Z$. By compactness of the closed disc, there
is an open neighbourhood $U$ of $x_0$ in $X$ so that $U \times D({\underline v}_{0},q)$
is disjoint from $Z$. By the triangle inequality, this  implies that 
 if $u\in U$ and $\underline v\in D({\underline v}_0,q-c)$, then $h(u,\underline v) >c$.
So $U\times D({\underline v}_0,q-c)$ is neighbourhood of $(x_0,{\underline v}_0)$ in  $h^{-1}(c,\infty)$.  
 
It follows that
\[
V = \{(x,{\underline v}_1,\dots {\underline v}_{p}) \in \RR^k \setminus Z(x) \mid d({\underline v}_{i},{\underline v}_{i+1}) < d({\underline v}_{i}, Z(x))/2 \}
\]
is open, and we can define a
deformation retraction from $V$ to $A$ by
\[
H_t(x,{\underline v}_{1},\dots {\underline v}_{i}, {\underline v}_{i+1}, \dots {\underline v}_{p})= (x,{\underline v}_{1},\dots ,{\underline v}_{i}, {\underline v}_{i+1} + t({\underline v}_{i}-{\underline v}_{i+1}), \dots {\underline v}_{p}).
\]
\end{proof}
\end{lemma}

The fibres of the projection 
 $p_Z:\Realization{K_{\bullet}(X,Z)} \to X$ are 
$\Realization{K_{\bullet}(\{x\},Z(x))}$ which are contractible by
lemma~\ref{le:pointspace}. We need a criterion that guarantees that
$p_{Z}$ is a weak homotopy equivalence. The following theorem is
proved in
section~\ref{sec:simplicial-spaces} below. 

\begin{theorem}
\label{th:hecriterion}
Let $X$ be a space, and $K_{\bullet}$ a simplicial subspace of
$X\times N_{\bullet}$ such that $K_{q}$ is open in $X\times N_{q}$
for all $q$. Let $\pi : \Realization{K_{\bullet}}\to X$ be the
projection. Suppose that $N_{\bullet}$ is contractible. Then $\pi$ is a weak homotopy equivalence. 
\end{theorem}

We are now ready to state the main result of this section.

\begin{theorem}
\label{th:cuts}
Let $(X,Z)$ be a critical pair, $X$ a metrizable space, and
assume that $Z(x)\subset \RR^k$ has measure 0 for each $x\in X$. Then  
the projection $p_Z:\Realization{K_{\bullet}(X,Z)} \to X$ 
is a weak homotopy equivalence.
\begin{proof}
Since $Z\subset X\times \RR^{k}$ is closed, $K_{q}(X,Z)$ is
an open subset of $K_{q}(X,\emptyset)=X\times N_{q}(\RR^{k})$.
The fibre $p^{-1}_{Z}(x)$ can be
identified with $\Realization{K_{\bullet}(x,Z(x))}$.
By  lemma~\ref{le:pointspace} both 
$p^{-1}_{Z}(X)$ and $N_{\bullet}(\RR^{k})=K_{\bullet}(\{x\},\emptyset)$
are weakly contractible, so theorem~\ref{th:hecriterion} applies.
\end{proof}
\end{theorem}

\section{Categories of embedded manifolds}
\label{sec:embedded-manifolds}
In this section we show how the abstract theory of 
section~\ref{sec:critical-pairs-cuts}
 applies to the theory of
embedded manifolds, and we define the $k$-tuple category of
manifolds in the Cartesian product of a euclidean space and a 
$k$-dimensional cube. We show that it deloops the category
of embedded manifolds considered in \cite{Galatius} and \cite{GMTW}.

\subsection{The space of embedded manifolds}
\label{sec:space-embedded-manifolds}

Following \cite{Galatius}, we consider the space of properly embedded
smooth $d$-manifolds without boundary in Euclidean $(d+n)$-space,
\[
\Psi_{d}(\RR^{d+n})=
\{W^{d}\subset \RR^{d+n}\mid \partial W=\emptyset, \text{ $W$ a closed subset}\}.
\]
We topologize $\Psi_d(\RR^{d+n})$ so that a sequence of manifolds
that leaves each compact set converges to the base point 
$\emptyset\in \Psi_{d}(\RR^{d+n})$, and so that manifolds close to 
$W$ are sections in a thin normal tube on a compact set.
More precisely, let $NW\subset \RR^{d+n}$ be
a normal tube and let $K\in \RR^{d+m}$ be
a compact subset. Let $r: NW \to W$ be the projection and let
$C^{\infty}$ be the set of smooth sections of $r$. We equip it
with the $C^{\infty}$-Whitney topology. For technical
reasons, we first chose a metric $\mu$ on the compact
Grassmannian manifold $G(d,n)$ of $d$-planes in $(d+n)$-space. For  
$s\in C^{\infty}(W,NW)$ and a compact set $K\subset \RR^{d+n}$
we define
\[
\Norm{s}_{K}=\sup_{x\in W\cap
  K}(\Realization{s(x)}+\mu(T_{x}W,ds(T_{x}W))\rlap{ ,}
\] 
where $\Realization{\cdot}$ is the norm in $\RR^{d+n}$ and $d$ denotes
the differential. 
The open neighbourhoods of $W\in \Psi_{d}(\RR^{d+n})$ are specified
by a pair $(K,\epsilon)$ with $K$ as above and $\epsilon>0$. Define
\[
\Gamma_{K,\epsilon} = \{s\in C^{\infty}(W,NW)\mid \Norm{s}_{K}<\epsilon\}.
\]
and define the corresponding neighbourhood of $W$ in
$\Psi_{d}(\RR^{d+n})$ to be
\[
{\cal N}_{K,\epsilon}(W)=\{V^{d}\mid V^{d}\cap K=s(W)\cap K
\text{ for } s\in \Gamma_{K,\epsilon}\}.
\]
 The neighbourhoods of
the empty  manifold are
\[
{\cal N}_{K}(\emptyset)=\{V\in \Psi_{d}(\RR^{d+n})\mid K\cap V=\emptyset\}.
\]

\begin{theorem}
\label{le:topology}
The sets ${\cal N}_{K,\epsilon}(W)$ form the a system of
neighbourhoods of a topology. Given $W$ and $K$
 the set ${\cal N}_{K,\epsilon}(W)$ is open,
 for sufficiently small $\epsilon>0$.
With this topology, $\Psi_{d}(\RR^{d+n})$ is a metrizable space.
\begin{proof}

We give a proof of this lemma in section~\ref{sec:topology}.

\end{proof}
\end{theorem}

We shall consider the subset $D^{k}_{d,n}\subset \Psi_{d}(\RR^{d+n})$
of manifolds where the projection on the first $k$  coordinates 
$f:W^{d}\to \RR^{k} $ is proper, or said a little differently,
let
\[
D^{k}_{d,n}=\{W^{d}\in \Psi_{d}(\RR^{d+n})\mid W^{d}\subset \RR^{k}
\times \inter({I}{}^{n+d-k})\},
\]
where $\inter (I{}^{N})$ is the open $N$-cube 
$(-1,1)^{N}\subset \RR^{N}$.

For a subset $S\subset\{1,\dots,k\}$ and $f:W\to \RR^{k}$, let
$f_{S}: W\to\RR^{S}$ be the projection onto the coordinates given
by $S$. If $\underline{a}\in \RR^{k}$, let 
$\underline{a}_S=(\underline{a}_{i})_{i\in S}\in \RR^{S}$ and define $Z(W)$ to be
the subset of vectors $\underline{a}\in \RR^{k}$ for which there
exists an $S$ such that $f_{S}$ is \emph{not} transversal to 
$\underline{a}_{S}$. If $W\in D^{k}_{d,n}$ then 
$\underline{a}\not\in Z(W)$ 
is the statement that 
$W^{d}\subset \RR^{k}
\times \mathop{int}{I}{}^{n+d-k}$ intersects all the affine
subspaces
\[
A(\underline{a},S)=\{\underline{x}\in \RR^{d+n}
\mid x_{i}=\underline{a}_{i} \text{ for $i\in S$}\}
\]
transversely. If we set
\[
W(\underline{a};S)=f_{S}^{-1}(\underline{a}_{S})=W\cap A(\underline{a},S)
\]
then
\[
W(\underline{a};S)\cap W(\underline{a};T)=
W(\underline{a};{S\cup T})  
\]
with transverse intersection, provided $\underline{a}\not\in Z(W)$.
By Sard's theorem each $Z(W)\subset \RR^{k}$ is 
closed and has measure 0. In order to apply the abstract theory of
section~\ref{sec:critical-pairs-cuts} we consider
\[
Z^{k}_{d,n}=\{(W,\underline{a})\in D^{k}_{d,n}\times \RR^{k}\mid 
\underline{a}\in Z(W)\}
\]
We must show that the pair $(D^k_{d,n},Z^k_{d,n})$ is a critical pair,
i.e.

\begin{proposition}
\label{prop:Zclosed}
$Z^{k}_{d,n}$ is a closed subset of $D^{k}_{d,n}\times \RR^{k}$.
\begin{proof}
We prove that the complement is open, so let
$(W,\underline{a})\not\in Z^{k}_{d,n}$. It suffices to find a
neighbourhood of $(W,\underline{a})$ in $D^{k}_{d,n}\times \RR^{k}$
such that for each $(V,\underline{b})$ in this neighbourhood,
$f: V \to \RR^{k}$ is transverse to $\underline{b}$. Indeed, the
argument below can be repeated for each 
$f_{S}:W\to \RR^{S}$. 

Since $f:W \to \RR^{k}$ is proper, hence closed, the singular values
is a closed subset of $\RR^{k}$. Let $D_{\epsilon}\subset \RR^{k}$ be 
an $\epsilon$-disc around $\underline{a}$ of regular values for $f$,
and set $K_{\epsilon}=D_{\epsilon}(\underline{a})\times 
I^{d+n-k}\subset \RR^{d+n}$. Choose $\delta>0$ so small that for
$s\in \Gamma_{K_{\epsilon/2},\delta}(s_{0})$, where $s_{0}:W\to NW$
denotes the zero section, one has
\begin{enumerate}[(i)]
\item for $x\in K_{\epsilon/2}\cap W$, $s(x)\in K_{\epsilon}$,
\item for $x\in K_{\epsilon}\cap W$, the differential $d(f\circ s)_x$
is surjective.  Since $f$ is the projection, $f\circ s=(s^{1},\dots,s^{k})$.
\end{enumerate}
For $(V,\underline{b})\in 
{\cal N}_{K_{\epsilon/2},\delta}(W)
\times D_{\epsilon/2}(\underline{a})$
we must show that $f:V \to \RR^{k}$ is transverse to $\underline{b}$.
Let $V=s(W)$ with $s\in \Gamma_{K_{\epsilon/2},\delta}(s_{0})$, and
let $y=s(x)\in V$ be an element of $f^{-1}(\underline{b})$.
By (i), $f\circ s(x)\in D_{\epsilon}(\underline{a})$, and by (ii),
\[
T_{x}(W)\xrightarrow{ds_{x}} T_y(V)\xrightarrow{df_{y}}\RR^{k}
\]
is surjective. But $ds_{x}$ is an isomorphism, so that
$df_{y}:T_{y}V\to \RR^{k}$ is surjective, and $\underline{b}$ is
a regular value.
\end{proof}
\end{proposition}

Since $\psi_{d}(\RR^{d+n})$ is metrizable by theorem~\ref{le:topology}
and $Z^{k}_{d,n}$ is closed by proposition~\ref{prop:Zclosed}
we can apply theorem~\ref{th:cuts} to get

\begin{corollary} 
\label{cor:regularPoints}
The projection
$p:\Realization{K_{\bullet}(D^{k}_{d,n},Z^{k}_{d,n})}\to D^{k}_{d,n}$
is a weak homotopy equivalence. 
\hfill
$\square$
\end{corollary}

S. Galatius in section 6 of~\cite{Galatius} determined the homotopy
type of $D^{k}_{d,n}$ by applying Gromov's theory of microflexible
sheaves to the sheaf $\Psi_{d}$ defined on open subsets 
$U\subset \RR^{d+n}$,
\[
\Psi_{d}(U)=\{W^{d}\subset U\mid \partial W^{d}=\emptyset,
\text{ $W^{d}$ a closed subset\}}.
\]
This is a microflexible sheaf in the terminology of \cite{Gromov},
see \cite{OR-W} for details, and the theory shows that ``scanning''
defines a homotopy equivalence
\[
D^{k}_{d,n}\simeq \mathrm{Map}((I^{d+n-k},\partial I^{d+n-k})\times
\RR^{k},(\Psi_{d}(\RR^{d+n}),\emptyset))
\]
(cf. section 4.2 of~\cite{Galatius}).

Let $U_{d,n}^\perp$ be the $n$-dimensional bundle over
the Grassmannian $G(d,n)$ of $d$-dimensional linear subspaces of 
$\RR^{d+n}$, consisting of pairs $(V,v)\in G(d,n)\times \RR^{d+n}$
with $v\perp V$. There is an obvious map
\[
q:U^{\perp}_{d,n}\to \Psi_{d}(\RR^{d+n}), \quad q(V,v)=V-v,
\] 
which extends to a map of the Thom space
$\mathrm{Th}(U^{\perp}_{d,n})$ into $\Psi_{d}(\RR^{d+n})$ since
$V-v$ leaves every compact subset as $v \to \infty$, and 
$\mathrm{Th}(U^{\perp}_{d,n})$ is the one point compactification of
$U^{\perp}_{d,n}$. Lemma 6.1 of~\cite{Galatius} shows that
\[
q:\mathrm{Th}(U^{\perp}_{d,n})\to \Psi_{d}(\RR^{d+n})
\]
is a homotopy equivalence. Combined with 
corollary~~\ref{cor:regularPoints} we get
\begin{theorem}
\label{th:delooping}
$
\Realization{K_{\bullet}(D^{k}_{d,n},Z^{k}_{d,n})}
\simeq \Omega^{d+n-k}\mathrm{Th}(U^{\perp}_{d,n}).
$
\end{theorem}

\subsection{The $k$-category of manifolds in a cube.}
\label{sec:manifolds-in-cube}
We begin with a smooth submanifold
\[
W^d_\epsilon \subset (-\epsilon, 1+\epsilon)^k \times 
\inter(I{}^{d+n-k})
\]
which is a closed subset of 
$(-\epsilon, 1+\epsilon)^k \times \RR^{d+n-k}$ and
intersects $[0,1]^k\times \inter(I{}^{d+n-k})$ orthogonally 
in a compact manifold with corners. More specifically,
let $\underline{v}=(v^1,\dots,v^{k})$ be a vertex of 
the $k$-dimensional cube $[0,1]^k$ and let $S\subset \{1,\dots,k\}$
be any subset. For $\epsilon > 0$, define
\begin{equation}
\begin{split}
\label{eq:boxes}
\Aset Sv\epsilon
&= \{\underline{x}\in \RR^{d+n}
\mid v^i-\epsilon < x^i < v^i + \epsilon, i \in S \},\\
A_S(\underline{v}) 
&= \{\underline{x}\in \RR^{d+n}
\mid x^i = v^i, i \in S \}.
\end{split}
\end{equation}
Notice that
\[
\Aset Sv\epsilon 
\cong
A_{S}(\underline{v}) \times \Jset Sv\epsilon,
\]
where $\Jset Sv\epsilon=\prod_{i\in S}(v^i-\epsilon,v^i+\epsilon)$.
We require $W^d_\epsilon$ to satisfy the following: 
\begin{condition}
$\text{ }$ 
\label{co:corners}
\begin{enumerate}[(i)]
\item
$W_\epsilon$ is transverse to $A_S(\underline{v})$ for all vertices $\underline{v}$ 
of $[0,1]^k$,
\item 
$W_\epsilon \cap \Aset  Sv\epsilon
=
(W_\epsilon \cap A_S(\underline{v})) \times 
\Jset Sv\epsilon,
$
\item $W_\epsilon$ is a closed subset of 
$(-\epsilon, 1+\epsilon)^k \times \RR^{d+n-k}$.
\end{enumerate}
\end{condition}

The intersection $W=W_\epsilon\cap ( [0,1]^k \times \inter(I{}^{d+n-k}))$
is a compact manifold with corners. In the terminology of~\cite{Laures},
it is a $\mathord{<}k\mathord{>}$-manifold embedded ``neatly'' in
$[0,1]^k \times \inter( I{}^{d+n-k})$ and equipped
with a $\mathord{<}k\mathord{>}$-collar.

The size $\epsilon$ of this collar is not part of the structure. We tacitly form the colimit where $\epsilon$ tends to 0.

Given $\underline{a}\leq \underline{b}$ in $\RR^k$, let 
\[
J(\underline{a},\underline{b})=\prod_{i=1}^{k}[a^i,b^i],\quad
J_ \epsilon(\underline{a},\underline{b})=\prod_{i=1}^{k}(a^i-\epsilon,b^i+\epsilon).
\]
We consider $d$-dimensional submanifolds 
\[
W_\epsilon^d \subset J_ \epsilon(\underline{a},\underline{b})
\times \inter(I{}^{d+n-k}),
\]  
which satisfy the analogue of the three conditions 
in~\ref{co:corners}. Since $W_\epsilon$ is constant on the collars
by condition (iii), it defines an element 
$\widehat{W}_\epsilon\in D^k_{d,n}$ with
\[
W_{\epsilon}=\widehat{W}_\epsilon \cap (J_ \epsilon(\underline{a},\underline{b})
\times \inter( I{}^{d+n-k}))
\]
upon extending the collars.
We are more interested in the space $N_1{\cal C}^k_{d,n}$
of all intersections
\[
W=W_\epsilon \cap (J(\underline{a},\underline{b})
\times \inter(I{}^{d+n-k})).
\]
By the remarks above, $N_1{\cal C}^k_{d,n}$ may be considered
a subspace of $K_1(D^k_{d,n},Z^k_{d,n})$, cf. definition~\ref{def:cutSets}.
There are $k$ directions to compose elements of $N_1{\cal C}^k_{d,n}$.
This turns ${\cal C}^k_{d,n}$ into a strict $k$-tuple category.
For $k=1$ this is the category of embedded cobordisms, examined
in~\cite{Galatius},~\cite{GMTW} for $n=\infty$. One can express the homotopy type
of $N_1{\cal C}^k_{d,n}$ in terms of classifying spaces. We sketch 
the result.

The cube $[0,1]^k$ is a $\mathord{<}k\mathord{>}$-manifold with
\[
\partial_i[0,1]^k=[0,1]^{i-1}\times \{0,1\}\times [0,1]^{k-i}.
\]
Let $W^d$ be a compact $d$-dimensional $\mathord{<}k\mathord{>}$-manifold
with an $\epsilon$-collar, cf. lemma 2.1.6 of~\cite{Laures} and let
\[
\mathrm{Emb}_\epsilon(W^d,[0,1]^k\times \inter( I{}^{d+n-k}))
\]
be the space of embeddings that maps the $\epsilon$-collar of $W^d$
to the  $\epsilon$-collar of $[0,1]^k\times\inter(I{}^{d+n-k})$
in the obvious linear fashion. Let
\[
\mathrm{Emb}(W^d,[0,1]^k\times \inter(I{}^{d+n-k}))=
\mathop{colim}_{\epsilon \to 0}
\mathrm{Emb}_\epsilon(W^d,[0,1]^k\times \inter(I{}^{d+n-k})).
\] 
For small values of $n$, this space might be empty, namely if the given
diffeomorphism type $W$ does not embed in codimension $n$.
The diffeomorphism group $\mathrm{Diff}(W)$  of the
collared $\mathord{<}k\mathord{>}$-manifold $W$ acts freely on
the embedding space, and the orbit
\[
B^k_{d,n}(W)=
\mathrm{Emb}(W^d,[0,1]^k\times \inter( I{}^{d+n-k}))/\mathrm{Diff}(W)
\]
is the set of collared $\mathord{<}k\mathord{>}$-submanifolds of
$[0,1]^k\times \inter(I{}^{d+n-k})$ diffeomorphic to $W$.
For $n=\infty$, 
\[
B^k_{d,n}(W)\simeq
B\mathrm{Diff}(W)
\]
by Whitney's embedding theorem.
Let $C(k)$ be the space of all $k$-cubes with edges parallel to the axes,
\[
C(k)=\{J(\underline{a},\underline{b})\in \RR^k\mid
\underline{a} <\underline{b}\}. 
\] 
More generally, set 
\[
C(k-l)=\{J(\underline{a},\underline{b})\}\mid
a^i=b^i \text{ for precisely $l$ indices}\}.
\]
The subspace of non-degenerate morphisms of ${\cal C}^k_{d,n}$
is homeomorphic to 
$\coprod C(k)\times B^k_{d,n}(W^d)$ with $W^d$ ranging over all
$\mathord{<}k\mathord{>}$-manifolds of dimension $d$ that
embeds in $[0,1]^k\times \inter( I{}^{d+n-k}))$.
If we intersect such an embedded  manifold with one of the
$(k-1)$-dimensional faces we get a non-degenerate object of
${\cal C}^k_{d,n}$, alias a non-degenerate morphism of 
${\cal C}^{k-1}_{d-1,n}$ etc. So we have
\begin{proposition}
There is a homotopy equivalence
\[
N_1{\cal C}^k_{d,n} \simeq 
\coprod_{l=0}^k\coprod_{W^{d-l}}C(k-l)\times
B^{k-l}_{d-l,n}(W^{d-l}).
\]
\end{proposition}

\subsection{The homotopy type of $B{\cal C}^k_{d,n}$}
\label{sec:homotopy-type}

The multi-nerve of a strict
$k$-tuple category is a $k$-dimensional simplicial space. The
associated diagonal simplicial space is denoted $N_\bullet{\cal C}^k_{d,n}$.
It is a subspace of $K_\bullet(D^k_{d,n};Z^k_{d,n})$. 

\begin{theorem}
\label{prop:he}
The inclusion
\[
N_\bullet{\cal C}^k_{d,n}\to K_\bullet(D^k_{d,n};Z^k_{d,n})
\]
induces a weak homotopy equivalence of realizations.
\end{theorem}
The proof will occupy the rest of this section,
but before we embark on it, we list its obvious consequence

\begin{theorem} 
\label{th:homotopyType}
The weak homotopy type of 
$B{\cal C}^k_{d,n}=\Realization{N_\bullet{\cal C}^k_{d,n}}$
is given by 
\[
B{\cal C}^k_{d,n}\simeq \Omega^{d+n-k}\mathrm{Th}(U^\perp_ {d,n}),
\]
where $U^\perp_{d,n}$ is the $n$-dimensional canonical vector bundle over
the Grassmannian $G(d,n)$ of $d$-planes in $\RR^{d+n}$. In particular,
we have  the weak homotopy equivalence 
\[
\Omega B{\cal C}^k_{d,n}\simeq B{\cal C}^{k-1}_{d,n}\quad\text{ for $1\leq k\leq d+n$}.
\]
\end{theorem}

\begin{remark}
The above theorem works equally well for oriented manifolds letting $G(d,n)$ be the
space of oriented $d$-planes, or more generally for the
category ${\cal C}^k_{d,n}(\theta)$ of manifolds with a $\theta$-structure in the sense 
of~\cite{GMTW}, section 5 or~\cite{MadsenWeiss}, section 2:
\begin{align}
B{\cal C}^k_{d,n}(\theta)&\simeq \Omega^{d+n-k}\mathrm{Th}(\theta^*U^\perp_ {d,n}),\\
\Omega B{\cal C}^k_{d,n}(\theta)&\simeq B{\cal C}^{k-1}_{d,n}(\theta)\quad\text{ for $1\leq k\leq d+n$}.
\end{align}
\end{remark}

The simplicial spaces $N_\bullet{\cal C}^k_{d,n}$ and $K_\bullet(D^k_{d,n};Z^k_{d,n})$
differ in two aspects. Elements of $N_\bullet{\cal C}^k_{d,n}$ intersect the facets of
the $k^2$ cubes orthogonally in small collars whereas elements of 
$K_\bullet(D^k_{d,n};Z^k_{d,n})$ are merely transversal to the facets. The second difference 
is that the elements of $K_\bullet(D^k_{d,n};Z^k_{d,n})$ are supported on manifolds that
are closed subsets of $\RR^k\times \inter(I{}^{d+n-k})$ while elements of
$N_\bullet{\cal C}^k_{d,n}$ are only subsets of an $\epsilon$-collar of the union
of cubes that is associated to the  element. The inclusion of $N_\bullet{\cal C}^k_{d,n}$
into $K_\bullet(D^k_{d,n};Z^k_{d,n})$ is by extending the manifold in the 
$\epsilon $-collar ``linearly''.

The numbers $d,n$ and $k$ will be constant in the following, and we shall
from now on drop the indices and simply write $(D,Z)$. We prove
theorem~\ref{prop:he} in two steps, first modifying elements of $K_ 1(D,Z)$ to
have orthogonal intersection with the facets, and second making elements 
affine outside the $\epsilon$-collar of the union of the $k^2$ cubes.

We say that an element $(W,\underline{a})\in K_0(D,Z)$ has orthogonal corner
structure (at $\underline{a})$ if for each $S\in \{1,\dots ,k\}$ 
\begin{equation}
\label{eq:orthogonal}
W\cap A_S(\underline{a},\epsilon))=W_S(\underline{a})\times 
J_{S}(\underline{a},\epsilon).
\end{equation}
for some $\epsilon>0$. Here we use  the notation of 
section~\ref{sec:space-embedded-manifolds}, and in particular
\[
J_{S}(\underline{a},\epsilon)=\prod_{i\in S}(a^i-\epsilon,b^i+\epsilon).
\]
Let $K^\perp_0(D,Z)\subset K_0(D,Z)$ be the subspace of elements
with orthogonal corner structure at $\underline{a}$.

\begin{lemma}
The inclusion $K^\perp_0(D,Z)\to K_0(D,Z)$ is a weak
homotopy equivalence.
\begin{proof}
Let $K_0(D,Z)(\underline{0})$ be the subspace of $K_0(D,Z)$
consisting of elements $(W,\underline{0})$. It is
a deformation retract via the deformation
$(W,\underline{a})\mapsto (W-t\underline{a},(1-t)\underline{a})$
as $0\leq t\leq 1$. Similarly, $K^\perp_0(D,Z)(\underline{0})$ is a deformation
retract of $K^\perp_0(D,Z)$, so it suffices to show that
\[
K^\perp_0(D,Z)(\underline{0})\to K_0(D,Z)(\underline{0})
\]
is a weak homotopy equivalence.

Given $(W,\underline{0})\in K_0(D,Z)(\underline 0)$, the projection
$f_S:W\to \RR^S$ is transversal to each point in
$J_{S}(\underline 0,\epsilon)$. Let $\lambda_\epsilon:\RR \to \RR$ be
a fixed smooth function subject to the following
requirements:
\begin{enumerate}[(i)]
\item $\lambda_\epsilon$ is weakly increasing and proper,
\item $\lambda_\epsilon(x)=x$ for $x\leq -\epsilon$ and
$x \geq \epsilon$,
\item $\lambda_\epsilon(x)=0$ for $x\in (-\epsilon/2,\epsilon/2)$.
\end{enumerate}
Let $\hat\lambda_\epsilon:\RR^k\times \inter(I{}^{d+n-k})\to
\RR^k\times \inter( I{}^{d+n-k})$ be the function that sends
$(x^1,\dots,x^k,y)$ to $(\lambda_\epsilon(x^1),\dots,\lambda_\epsilon(x^k),y)$.
By the transversality assumption,
\[
(\hat\lambda_\epsilon)^*(W)=
\{(x,y)\in \RR^k\times \inter( I{}^{d+n-k})\mid 
\hat\lambda_\epsilon(\underline{x},\underline{y})\in W\}
\]
is a submanifold, and one easily checks that
\begin{align*}
(\hat \lambda_\epsilon)^*(W)_S(\underline{0})
&=W_S(\underline{0})\text{ , and}\\
(\hat\lambda_\epsilon)^*(W)_S \cap A(\underline{0},\epsilon/2)
&=W_S(\underline{0}) \times J_{S}(\underline{0},\epsilon/2).
\end{align*}
Thus $(\hat\lambda_\epsilon)^*(W)$ has an orthogonal 
corner structure at $\underline{0}$.
There is a path $(\hat\lambda_\epsilon^t)^*(W)$
from $W$ to $(\hat\lambda_\epsilon)^*(W)$ given by the
function
\[
\hat\lambda_\epsilon^t(x)=(1-t)x+t\hat\lambda_\epsilon(x)
\]
Note that the entire path $(\hat\lambda_\epsilon^t)^*(W)$
is in $K^\perp_0(D,Z)(\underline 0)$ when $W\in
K^\perp_0(D,Z)(\underline 0)$

The number $\epsilon>0$ depends on the given
$W\in K_0(D,Z)(\underline 0)$, but can be kept constant in a neighbourhood 
of $W$; this follows from Proposition~\ref{prop:Zclosed} and its proof.
Thus for each compact subset $C\subset K_0(D,Z)(\underline 0)$ there is an
$\epsilon=\epsilon(C)>0$ and a diagram
\[
\xymatrix{
C\times I\ar[r]^{h_t}\ar@{}[d]^-{\displaystyle\cup}
&K_0(D,Z)(\underline 0)\ar@{}[d]^-{\displaystyle\cup}\\
C\cap K_0^\perp(D,Z)(\underline 0)\ar[r]^{h_t^\perp}
&K_0^\perp(D,Z)(\underline 0)
} 
\]
which for $t=0$ is the inclusion diagram and such that
$h_1:C\to K_0^\perp(D,Z)(\underline 0)$. It follows that all
relative homotopy groups
$\pi_i(K_0(D,Z)(\underline 0),K_0^\perp(D,Z)(\underline 0))=0$.
\end{proof}
\end{lemma}

An element $(W,\underline{a}_0,\dots,\underline{a}_r)\in K_r(D,Z)$ 
gives rise to a subdivision of the cube 
$C(\underline{a},\underline{b})$
into $(r-1)^{k}$ sub cubes. Define $K^\perp_r(D,Z)$ to be the subspace
of $K_r(D,Z)$ of elements $(W,\underline{a}_0,\dots,\underline{a}_r)$ where
$(W,\underline{v})\in K_0^\perp(D,Z)$ for each sub cube vertex $\underline{v}$.

\begin{corollary}
\label{cor:perp}
The inclusion $K^\perp_r(D,Z)\to K_r(D,Z)$ is a weak
homotopy equivalence, $r\geq 0$.
\begin{proof}
Apply the homotopy constructed in the proof 
of the the previous lemma simultaneously to
$(W,\underline{v})\in K_0(D,Z)$ for all vertices 
$\underline{v}$ in the sub cubes.
\end{proof}
\end{corollary}

Next we consider an embedding 
$N_\bullet {\cal C}^{k}_{d,n} 
\to K^{\perp}_\bullet(D^{k}_{d,n},Z^{k}_{d,n})$ 
where
$N_{\bullet}{\cal C}^{k}_{d,n}$ is the diagonal simplicial space
of the $k$-dimensional multi nerve. We first describe the image of
the embedding.  

For $\underline{a}\leq \underline{b}$, remember the notation
\[
J(\underline{a},\underline{b})=\prod [a^{i},b^{i}], \quad
J_{\epsilon}(\underline{a},\underline{b})=\prod (a^{i}-\epsilon,b^{i}+\epsilon).
\]

An element $(W,\underline a_0,\dots,\underline a_r) \in  N_r{\cal C}^k_{d,n}$
is by definition the intersection of $W_\epsilon\subset J_ \epsilon(\underline a_0,\underline a_r)
\times \RR^{d+n-k}$ with $ J(\underline a_0,\underline a_r)
\times \RR^{d+n-k}$ , where $W_\epsilon$ is a product in an $\epsilon$-collar of the
boundary of  $ J(\underline a_0,\underline a_r)$, cf. \S (\ref{sec:manifolds-in-cube}).
One defines $\hat W_\epsilon \subset \RR^k \times \inter(I^{d+n-k})$ by extending
the $\epsilon$-collars, and obtain the embedding
\[
 N_r{\cal C}^k_{d,n} \hookrightarrow  K^{\perp}_r(D^{k}_{d,n},Z^{k}_{d,n})
\]
by sending $(W,\underline a_0,\dots,\underline a_r)$ to
$(\hat W_\epsilon,\underline a_0,\dots,\underline a_r)$. There is a retraction by 
intersecting $(\hat W,\underline a_0,\dots,\underline a_r)\in   K^{\perp}_\bullet(D^{k}_{d,n},Z^{k}_{d,n}) $
with an $\epsilon$-collar of the boundary of  $ J_\epsilon(\underline a_0,\underline a_r)$
for some $\epsilon$. The elements of 
$ N_r{\cal C}^k_{d,n}$ and $ K^{\perp}_\bullet(D^{k}_{d,n},Z^{k}_{d,n})$ agree on
 $ J_ \epsilon(\underline a_0,\underline a_r)
\times \RR^{d+n-k}$ but may differ on
the complement 
$ (\RR^k \setminus J_ \epsilon(\underline a_0,\underline a_r))
\times \RR^{d+n-k}$. 

\begin{theorem}
\label{th:degreewise}
For each $r$, the inclusion
\[
 N_r{\cal C}^k_{d,n} \hookrightarrow  K^{\perp}_r(D^{k}_{d,n},Z^{k}_{d,n})
\]
is a homotopy equivalence.
\begin{proof}
Given 
$ (\hat W,\underline a_0,\dots,\underline a_r)\in   K^{\perp}_r(D^{k}_{d,n},Z^{k}_{d,n})$
we must specify a curve from this element to $ N_r{\cal C}^k_{d,n}$, independent of 
$\hat W$ and depending continuously on $(\underline a_0,\dots,\underline a_r)$.
The idea is to expand the outside collar of  $ J_ \epsilon(\underline a_0,\underline a_r)$
without moving the collar of size $\epsilon/2$.

If $\underline a_0 < \underline a_r$, we can scale by an affine map to 
$\underline a_0=(0,\dots,0)$ and $\underline a_r=(1,\dots,1)$, so that 
$ J_ \epsilon(\underline a_0,\underline a_r)=(-\epsilon,1+\epsilon)^k$.
The degenerate situation $\underline a_0=\underline a_r$ is similar but
easier.

We introduce the following notation. For $0<\epsilon\leq \mu$ let $D^\perp_{\mu,\epsilon}$
denote the space of submanifolds 
\[
W\subset (-\mu,1+\mu)^k\times \inter (I^{d+n-k})
\]
which are closed as subsets of $ (-\mu,1+\mu)^k\times \RR^{d+n-k}$
and such that the restriction
\[
\mathrm{Res}^\mu_\epsilon(W)=W\cap (-\epsilon,1+\epsilon)^k\times \inter(I^{d+n-k})
\]
satisfies the three conditions of~(\ref{co:corners}).

We define the embedding by extending the outside collar
\[
\phi_\epsilon:D^\perp_{\epsilon,\epsilon}\to  D^\perp_{\infty,\epsilon}
\]
as follows. Choose a standard diffeomorphism
\[
\phi_\epsilon: (-\epsilon,1+\epsilon)\to \RR.
\]
The $\phi^k_\epsilon\times \mathrm{id}$ is a diffeomorphism from 
$ (-\epsilon,1+\epsilon)^k\times \inter(I^{d+n-k})$
to $ \RR^k\times \inter(I^{d+n-k})$ and
\[
\hat\phi_\epsilon(W):=(\phi_\epsilon^k\times \mathrm{id})(W)
\]
We have 
\begin{equation}
\label{eq:required-homotopy}
\mathrm{Res}_\epsilon^\infty\circ \hat \phi_\epsilon \simeq \mathrm{id}\text{ and }
 \hat \phi_\epsilon\circ\mathrm{Res}_\epsilon^\infty\simeq \mathrm{id}
\end{equation}
The first equivalence is obvious. The second homotopy is given in the following
way. For $t\geq \epsilon$, let $\rho_t:\RR\to \RR$ be the affine map with
$\rho_t(-t)=- \epsilon$ and $\rho_t(1+t)=1+\epsilon$. Set 
$\psi_t=\rho_t^{-1 }\circ \phi_\epsilon\circ \rho_t$ and consider
\[
\hat \psi_t:D^\perp_{t,\epsilon}\to D^\perp_{\infty,\epsilon}
 \]  
Since $\psi_t$ is constant on subintervals of $(-t,1+t)$ that tends to
$(-\infty,\infty)$ for $t\to \infty$, the composition
\[
\hat \psi_t\circ \mathrm{Res}^\infty_t:D^\perp_{\infty,\epsilon}\to D^\perp_{\infty,\epsilon}
\]
has limit $\hat \psi_\infty=\mathrm{id}$ at $t\to \infty$. Thus 
$W_t=\hat \psi_t\circ \mathrm{Res}^\infty_t(W)$ is a curve from 
$\hat \phi_\epsilon\circ \mathrm{Res}^\infty_\epsilon(W)$ to $�₁�W$
at $t\in [\epsilon,\infty)$. This is the required homotopy in (\ref{eq:required-homotopy}).
\end{proof}
\end{theorem}

A map of simplicial spaces $X_\bullet\to Y_\bullet$ which is a degreewise
homotopy equivalence induces a (weak) homotopy equivalence of 
topological realizations, so theorem \ref{prop:he} is a consequence of theorem
\ref{th:degreewise}.

It is sometimes more convenient to work with the simplicial space of
discrete cut sets $K^\delta_\bullet(D^k_{d,n},Z^k_{d,n})$ rather than with
$K_\bullet(D^k_{d,n},Z^k_{d,n})$ where the cut points can move
continuously. 
We let $ N_\bullet^\delta{\cal C}^k_{d,n}$ be the set $N_\bullet{\cal C}^k_{d,n}$,
but re-topologized as a subset of $ K^{\delta}_\bullet(D^{k}_{d,n},Z^{k}_{d,n})$.
This gives the diagram of simplicial spaces
\begin{equation}
\label{eq:continuous-to-discrete}
\xymatrix{
N_\bullet^\delta{\cal C}^k_{d,n}\ar[r]\ar[d]&
 K^{\delta}_\bullet(D^{k}_{d,n},Z^{k}_{d,n})\ar[d]\\
N_\bullet{\cal C}^k_{d,n}\ar[r]&
 K_\bullet(D^{k}_{d,n},Z^{k}_{d,n})
}
\end{equation}
where the horizontal arrows are inclusions. In the proof above of theorem
\ref{th:degreewise}, we did not move the cut points, so the same argument gives

\begin{addendum}
The map $ N_\bullet^\delta{\cal C}^k_{d,n} \rightarrow  K^{\delta}_\bullet(D^{k}_{d,n},Z^{k}_{d,n})$
induces a weak homotopy equivalence.
\end{addendum}
The right hand vertical map in (\ref{eq:continuous-to-discrete})
is a weak homotopy equivalence by theorem  \ref{th:degreewise}, so
 (\ref{eq:continuous-to-discrete}) is a diagram of weak
homotopy equivalences.

\section{Simplicial spaces}
\label{sec:simplicial-spaces}
 The purpose of this section is to prove some facts about 
the realizations of simplicial spaces that we need for the proof
of theorem~\ref{th:hecriterion}. We
construct a regular neighbourhood of a degreewise open subset,
and apply this to give a criterion that ensures that
a map with contractible fibres is a homotopy equivalence.

\subsection{The second derived neighbourhood of simplicial spaces}
\label{sec:second-deriv-neighb}

This section contains a version of the regular neighbourhood theorem 
for a pair of simplicial spaces. 
Suppose that $Y_{\bullet}\subset X_{\bullet}$ is a simplicial
space with a simplicial subspace, and assume for convenience that
the spaces are degreewise compactly generated (\cite{Steenrod}). 
This is the case
for example if $X_{\bullet}$ consists of metrizable spaces.
Suppose that in each degree
the inclusion is the inclusion of an open subspace. As the special case of
simplicial sets shows, we cannot expect that the induced map
of realizations is the inclusion of an open subset. 
For a discussion of this, see \cite{Lewis}.

Let the open star $\mathrm{St}(X_{\bullet},Y_{\bullet})$ be the 
union of all open simplices $t$ in $\Realization{X_{\bullet}}$ such
that at least one vertex of $t$ is contained in $Y_{0}$.
We consider the vertex maps
$v_{i}\colon X_{n}\to X_{0}, (0\leq i\leq n)$ 
that maps a simplex to its $i^{\text{th}}$
vertex (induced by the inclusion $[0]\ni 0\mapsto i \in [n]$).

\begin{lemma} 
\label{le:open}
Assume that $Y_{n}$ is an open subset of $X_{n}$.
Then the open star $\mathrm{St}(X_{\bullet},Y_{\bullet})$ is an open
subset of $\Realization {X_{\bullet}}$.
\begin{proof}

Let $\phi_{n} \colon X_{n}\times \Delta^{n}\to \Realization{X_{\bullet}}$
be the characteristic map. By the definition of the topology of
the realization, it is enough to show that for every $n$,
the set $\phi_{n}^{-1}(\mathrm{St}(X_{\bullet},Y_{\bullet}))$ is open.

Let $Y^{\prime}_{n} = \{ x\in X_{n}\mid {\underline v}_{i}(x)\in Y_{0} \text{ for
  some $i$}\}$. 
Since $Y_{0}$ is open in $X_{0}$ and each $v_i$ is continuous
this is an open subset of $X_{n}$.
If $\alpha : [k] \to [n]$
is a morphism in the simplicial category, then
$(\alpha^*)^{-1}(Y^{\prime}_{k}) \subset Y^{\prime}_{n}$.
By definition,
\[
St(X_{\bullet},Y_{\bullet})=
\bigcup_{k} \phi_{k}(Y^{\prime}_{k}\times \mathrm{int}(\Delta^{k})).
\]

A point in $\Realization{X_{\bullet}}$ is uniquely represented by some
$(y,t)\in X_{n}\times \mathrm{int}(\Delta^{n})$, so this union is 
actually a disjoint union. Moreover, if $(x,s)\in X_{k}\times
\Delta^{k}$ 
is any other
representative of the same point, there is an injective morphism 
$\alpha\colon [k]\to [n]$ such that $\alpha_*(t)=s$ and
$\alpha^*(x)=y$ (\cite{May}, lemma 14.2).

Given a point $(y,t)\in Y_{k}^{\prime}\times
\mathrm{int}(\Delta^{k})$  we have to show the following property: 
If $(x,s)\in X_{n}\times
\Delta^{n}$ represents the same point as $(y,t)$ in $\Realization{X_{\bullet}}$, 
then $(x,s)$ is an inner point
of $\phi_n^{-1}(\mathrm{St}(X_{\bullet},Y_{\bullet}))$.
Let $\alpha\colon [k]\to [n]$ be as above. If $k=n$, the openness
follows because $Y_{n}^{\prime}$ is an open set, so we can assume that
$k<n$, and $s \in \partial \Delta^n$.

We claim that there is an open neighbourhood 
$U\subset \Delta^n$ of $s$, such that 
$(\alpha^*)^{-1}Y_k^\prime \times U$ is an an open set contained  
in $\phi_n^{-1}(\mathrm{St}(X_{\bullet},Y_{\bullet}))$. By induction on
$n$, we can assume that there is an open neighbourhood $V$ of 
$s$ in $\partial \Delta^n$, so that 
\[
(\alpha^*)^{-1}Y_k^\prime \times V \subset 
\phi_n^{-1}(\mathrm{St}(X_{\bullet},Y_{\bullet}))\cap X_n \times \partial \Delta^n.
\]
 
On the other hand, $(\alpha^*)^{-1}Y_k^\prime\times \mathrm{int}(\Delta^n)
\subset Y_n^\prime \times \mathrm{int}(\Delta^n) \subset 
\phi_n^{-1}(\mathrm{St}(X_{\bullet},Y_{\bullet}))\cap X_n \times \mathrm{int}(\Delta^n)$,
so that we can find the wanted neighbourhood $U$ by choosing it
as an arbitrary open neighbourhood of $s$ in $\Delta^n$ such that
$U\cap \partial \Delta^n \subset V$. 
 \end{proof} 
\end{lemma}

The open star construction gets better after subdivision. 
We remember that the subdivision $\Sd{}X_{\bullet}$ is
the nerve of the topological category of simplices of $X_{\bullet}$;
it has objects $([n],x)$ with $x\in X_{n}$ and a morphism from
$([n],x)$ to $([m],y)$ is a morphism $\alpha:[n]\to [m]$ with
$\alpha^{*}(y)=x$. 

Let $x\in \Sd{}X_{n}$ be
$x=([N_{0}]\to \dots \to [N_{n}],z\in X_{N_{n}})$. Its
$i^\text{th}$ vertex is $v_{i}(x)=([N_{i}],\beta^{*}(z))$, with
$\beta:[N_{i}]\to\dots\to [N_{n}]$. If
$v_{n}(x)=([N_{n}],z)\in \Sd{}Y_{0}$, i.e. if $z\in Y_{N_{n}}$ then
$x\in \Sd{}Y_{N_{n}}$. The pair 
$(Z_{\bullet},T_{\bullet})=(\Sd{}X_{\bullet},\Sd{}Y_{\bullet})$ thus has
the following

\vspace{1ex}
{\em Property I:}
A point $z\in Z_{n}$ is contained in $T_{n}$ if and only if
$v_{n}(z)\in T_{0}$.

\begin{remark}
Suppose that the pair $Z_{\bullet},T_{\bullet}$ has property I, and that
$z\in (\Sd{}Z)_{n}$ has the property that some vertex 
$v_{i}(z)\in (\Sd{}T)_{0}$. Then  
the last vertex $v_n(z)\in T_{0}\subset Z_{0}$.
\end{remark}

\begin{lemma}
\label{le:heopennbh}
Let $X_\bullet$ be a degreewise compactly generated simplicial space.
Suppose that for every $i$ the inclusion $Y_{i}\subset X_{i}$ 
is an open embedding, 
and that the pair $(X_\bullet,Y_\bullet)$ has Property I. Then,
the inclusion $\Realization{Y_\bullet}\subset \mathrm{St}(X_\bullet,Y_\bullet)$
is a homotopy equivalence.
\begin{proof}
We consider two natural maps associated to a subdivision of a simplicial
set $X_\bullet$. The first one is the the subdivision map
\[
s_{{X}}\colon \Realization{\Sd{}X_{\bullet}}\to \Realization{X_{\bullet}}.
\]
This is the unique map that is affine on simplices, and 
sends a vertex $a\in X_{n} \subset
Sd{}X_{0 }$ to the barycenter of the simplex represented by $a$ in
$\Realization{X_{\bullet}}$. The subdivision map is not a
simplicial map, but it is a homeomorphism
if $X_{\bullet}$ is assumed to be degreewise compactly generated. 
To see that $s_X^{-1}$ is continuous, consider for each $N$ 
the diagram 
\[
\xymatrix@C=3cm{
\coprod_{[N_{0}]\rightarrow \dots 
\rightarrow[N_{r-1}]
\rightarrow[N]} \Delta^{r}\times
X_{N}\ar[r]^-{\phi_{\Sd{}X}}\ar[d]^{f}
&\Realization{\Sd{}X_{\bullet}}
\ar[d]^{s_{X}}\\
\Delta^{N}\times X_N
\ar[r]^-{\phi_{X}}
\ar[ru]^{s_X^{-1}\phi_{X}}&
\Realization{X_{\bullet}}{\rlap{\ }}.  
}
\]
The vertical map $f:\Delta^r\times X_N\to \Delta^N\times X_n$,
associated with $[N_0]\to [N_1]\to \dots \to [N_r]$, ($N_r=N$), is the
product of the identity on $X_n$ and the following map  
$\Delta^r\to \Delta^N$: each $[N_i]\to [N_r]$ induces a 
simplicial $\Delta^{N_i}\xrightarrow{l_i} \Delta^{N_r}$;
let $b_i$ be the barycenter of the image $l_i(\Delta^{N_i})$.
Then $f:\Delta^r\to \Delta^N$ is the affine map that takes the
$i^\text{th}$ vertex in $\Delta^r$ to $b_i$.

The inverse of $s_{X}$ is continuous if and only if 
$(s_{X})^{-1}\phi_{X}$ is continuous. 
But this follows because if $X_N$ is compactly generated, 
then $f$ is an identification map (a surjection map,
where the target has the quotient topology), since
it is the product of a compact identification map with the identity
on a compactly generated space (\cite{Steenrod}).

The second map we consider is a ``first vertex map''. 
It is a simplicial map given in the following fashion. 
Let 
\[
[N_{0}] \xrightarrow{\alpha_{0}} [N_{1}] \xrightarrow{\alpha_{1}}
\dots \xrightarrow{\alpha_{n-1}}[N_{n}]
\] 
be an $n$-simplex in the subdivision.
The sequence determines a map in the simplicial category
$\alpha \colon [n]\to [N_{n}]$ by defining $\alpha(i)\in [N_{n}]$ to
be the image under the iterated maps of the first vertex
$0 \in [N_{n-i}]$, that is 
$\alpha(i)=\alpha_{n-1}\circ \alpha_{n-2}\circ\dots\circ
\alpha_{n-i}(0)$. The first vertex map is the continuous simplicial map
\[
L_{X}\colon  \Sd{}X_{\bullet} \to X_{\bullet}\rlap{ ,}\quad 
L_{X}(x,\alpha_{0},\alpha_{1},\dots,\alpha_{n-1}) =\alpha^{*}(x)\rlap{\ .}
\]

Let $z=(x,\alpha_{0},\alpha_{1},\dots, \alpha_{n-1})\in 
\mathrm{Sd}X_{\bullet}$.  
If $v_{i}(z)\in \mathrm{Sd}Y_{\bullet}$ for some $i$, then $L_{X}(z)\in Y_{n}$.
The condition means that 
$(\alpha_{n-1}\circ \dots \circ\alpha_{i})^{*}(x)\in Y_{N_{i}}$. 
Since 
$Y_{\bullet}$ is a simplicial subset, it follows that
$v_{n}(L_{X}(x))=(\alpha_{n-1}\circ\dots \circ \alpha_{0}\circ v_{0})^{*}(x)\in
Y_{0}$. From property I it follows that $L_{X}(x)\in Y_{n}$.
We conclude that
$\Realization{L_{X}}(\mathrm{St}(\Sd{}X_{\bullet},\Sd{}Y_{\bullet}))\subset
\Realization{\Sd{}Y_{\bullet}}$

The first vertex map does not induce a homeomorphism, but
the two maps $\Realization{L}, s_{X}\colon \Realization{\Sd{}X_{\bullet}}
\to\Realization{X_{\bullet}}$ are homotopic. Indeed, if 
$(z,t)\in \Sd{}X_{n}\times \Delta^{n}$ represents a point 
$\phi_{n}(z,t)\in \Realization{\Sd{}X_{\bullet}}$, then
$z$ is given by $x\in X_{N_{n}}$ and a sequence of 
morphisms $\{\alpha_{i}\}$
in the simplicial category. We compute
\begin{align*}
\Realization{L_{X}}(\phi_{n}(z,t))
&=\phi_{N_{n}}(x,\alpha_{*}t)\\
s_{X}(\phi_{n}(z,t))&=\phi_{N_{n}}(x,A(t)).
\end{align*}
where $A:\Delta^{n}\to \Delta^{N_{n}}$ is an affine map,
depending on the sequence of morphisms $\{\alpha_{i}\}$.
A homotopy $H_{s}$ from $\Realization{L_{X}}$ to $s_{X}$ is given by 
\[
H_{n}(\phi_{N_{n}}(z,t))=
\phi_{N_{n}}(x,((1-s)\alpha_{*}(t)+sA(t))).
\]

We can restrict the two maps from 
$\Realization{\Sd{}X_\bullet}$ to 
$\mathrm{St}(\Sd{}(X)_\bullet,\Sd{}(Y)_{\bullet})$.
When we do this for $L_X$ we get a commutative diagram

\[
\xymatrix@C=1,5cm{
\Realization{Y_{\bullet}}\ar[r]^{s_{Y}^{-1}}&
\Realization{\Sd{}Y_\bullet}\ar[r]^-{i}\ar[d]^{L_{Y}}&
\mathrm{St}(\Sd{}X_{\bullet},\Sd{}Y_{\bullet})
\ar@{}[r]|-{\displaystyle\subseteq}\ar[d]^{L_X}&
\Realization{\Sd{}X_\bullet}\ar[d]^{L_{X}}\\
&
\Realization{Y_\bullet}\ar@{=}[r] &
\Realization{Y_\bullet}\ar[r] &
\Realization{X_\bullet}\rlap{ ,}
}
\]
and it suffices to show that 
the map $i\circ s_{Y}^{-1}: \Realization{Y_{\bullet}}\to
\mathrm{St}(\Sd{}X_{\bullet},\Sd{}Y_{\bullet})$ has the homotopy
inverse $L_{X}:\mathrm{St}(\Sd{}X_{\bullet},\Sd{}Y_{\bullet})\to
\Realization{Y_{\bullet}}$.

The first vertex map $L_{Y}$ is homotopic to $s_{Y}$, so
that the composition
\[
\Realization{Y_{\bullet}}\xrightarrow{s_{Y}^{-1}}
\Realization{\Sd{}Y_{\bullet}}\xrightarrow{i}
\mathrm{St}(\Sd{}X_{\bullet},\Sd{}Y_{\bullet})\xrightarrow{L_{X}}
\Realization{Y_{\bullet}}
\]
is homotopic to the identity. 
Composing the other way, we obtain
a commutative diagram 
\[
\xymatrix{
\mathrm{St}(\Sd{}X_{\bullet},\Sd{}Y_{\bullet})
\ar[r]^-{L_{X}}
\ar@{}[d]|-{\displaystyle\bigcap}&
\Realization{Y_{\bullet}}
\ar[r]^-{s_{Y}^{-1}}
\ar@{}[d]|-{\displaystyle\bigcap}&
\Realization{\Sd{}Y_{\bullet}}
\ar[r]^-{i}
\ar@{}[d]|-{\displaystyle\bigcap}&
\mathrm{St}(\Sd{}X_{\bullet},\Sd{}Y_{\bullet})
\ar@{}[d]|-{\displaystyle\bigcap}
\\
\Realization{\Sd{}X_{\bullet}}\ar[r]^-{L_X}&
\Realization{X_{\bullet}}\ar[r]^-{s_X^{-1}}&
\Realization{\Sd{}X_{\bullet}}\ar@{=}[r]&
\Realization{\Sd{}X_{\bullet}}.
}
\]
The composite of the lower row $s_{X}^{-1}L_{X}$ is homotopic
to the identity by the homotopy 
$s_{X}^{-1}H_{s}$. To finish the proof of the lemma, we have to
argue that this homotopy preserves the subspace
$\mathrm{St}(\Sd{}X_{\bullet},\Sd{}Y_{\bullet})$.
This is equivalent to the statement that if 
$(z,t)\in \Sd{}X_{n}\times \mathrm{int}(\Delta^{n})$ 
represents a point 
$\phi_{n}(z,t)\in \mathrm{St}(\Sd{}X_{\bullet},\Sd{}Y_{\bullet})\subset
\Realization{X_{\bullet}}$, then
\[
H_{s}(\phi_{n}(z,t))\in
s_{X}\mathrm{St}(\Sd{}X_{\bullet},\Sd{}Y_{\bullet})
\subset \Realization{X_{\bullet}}. 
\]

We examine $\phi_{n}^{-1}(\Realization{Y_{\bullet}})$ 
and $\phi_{n}^{-1}(\mathrm{St}(\Sd{}X_{\bullet},\Sd{}Y_{\bullet}))$.
We are assuming property I, so for a fixed $x\in X_{N}$,  
there will be some $k$, $-1\leq k\leq n$ such that
$v_{i}(x)\in \Sd{}Y_{0}$ for $i \leq k$, and
$v_{i}(x)\not\in \Sd{}Y_{0}$ for $i > k$.
This means that $\{t\in \Delta^{n}\mid (x,t)\in
\phi_{n}^{-1}(\Realization{Y_{\bullet}})\}$ will be the 
convex span of the vertices $\{v_{i}\mid 0 \leq i \leq k\}$,
which is either the empty set (in case $k=-1$), 
or a sub-simplex of $\Delta^{n}$.

It follows that
$(x,t)\in
\phi_{N}^{-1}(\mathrm{St}(\Sd{}X_{\bullet},\Sd{}Y_{\bullet}))$
if and only if there is an $i\leq k$ such that $t_{i}>t_{j}$
for $k < j \leq n$.

Let $(z,t)\in (\Sd{}X_{n}\times \mathrm{int}(\Delta^{n}))$ 
represent a point in
$\mathrm{St}(\Sd{}X_{\bullet},\Sd{}Y_{\bullet})\subset
\Realization{X_{\bullet}}$.
If $z\in \Sd{}X_{r}$ is represented by
$[N_{0}] \xrightarrow{\alpha_{0}} \dots \xrightarrow{\alpha_{n-1}} [N_{n}]$ 
together with $x\in X_{N_{n}}$, 
then  the image of $(z,t)$ under the homotopy is represented by a line
segment in $\{x\} \times \Delta^{N_{n}}\subset X_{N_{n}}\times
\Delta^{N_{n}}$ which connects $(x,\alpha_{*}(t))$ to
$(x,A(t))$, where $s_{X}(\phi_{n}(z,t))=(x,A(t))$. By definition,  
$(x,A(t))\in
s_{X}\mathrm{St}(\Sd{}X_{\bullet},\Sd{}Y_{\bullet})$,
so there exists some $i\leq k$ such that 
$(\alpha_{*}(t))_{i}>(\alpha_{*}(t))_{j}$ for all $j > k$.
We also know that $L_{X}(\phi_{n}(z,t))\in \Realization{Y_{\bullet}}$,
so $A(t)_{j}=0$ for $j<k$. It follows that any convex combination 
$u=(1-s)\alpha_{*}(t)+sA(t)$ with $s<1$ also satisfies that
$u_{i}>u_{j}$ for all $j > k$, so that 
\[
H_{s}(\phi_{n}(z,t))=\phi_{N_{n}}(x,u)\in
s_{X}\mathrm{St}(\Sd{}X_{\bullet},\Sd{}Y_{\bullet})
\]  
for $0\leq s \leq 1$.
\end{proof}
\end{lemma}

Let $\Sd^2X_\bullet=\Sd {}(\Sd {} X_{\bullet})$ be the second barycentric subdivision. The following theorem is an immediate consequence of our work.

\begin{theorem}
\label{th:secondDerived}
 $\Realization{\mathrm{St}(\Sd^{2}X_{\bullet},\Sd^{2}Y_{\bullet})}$ is an open set in  
$\Realization{\Sd^{2}X_{\bullet}}$.
It contains $\Realization{(\Sd^{2}Y)_{\bullet}}$, and the inclusion of this subspace 
in $\Realization{\mathrm{St}(\Sd^{2}X_{\bullet},\Sd^{2}Y_{\bullet})}$ is a weak
homotopy equivalence.
\begin{proof}
By lemma~\ref{le:open} the star is an open subset.  
We have checked that $\Sd{}X_{\bullet}$ has property I, so the
theorem
follows from  lemma~\ref{le:heopennbh}.
\end{proof}
\end{theorem}

We conclude the section with an application of 
theorem~\ref{th:secondDerived}.

Let $f:X_{\bullet}\to Y_{\bullet}$ be a map of simplicial, degreewise
metrizable spaces.
Suppose that $X^{\alpha}_{\bullet}\subset X_{\bullet}$ 
and $Y^{\alpha}_{\bullet}\subset Y_{\bullet}$ are 
families of degreewise open subspaces, indexed by the 
same set $A$ such that 
$f(\Realization{X_{\bullet}^{\alpha}}) \subset 
\Realization{Y_{\bullet}^{\alpha}}$ 
for $\alpha\in A$ and $X_\bullet=\cup X_\bullet^ \alpha$,
$Y_\bullet=\cup Y_\bullet^ \alpha$. 

\begin{theorem}
\label{th:opencover}
Suppose that for each finite subset $I\subset A$ 
the restriction
\[
f \colon \bigcup_{\alpha\in I}\Realization{X_{\bullet}^{\alpha}} \subset 
\bigcup_{\alpha \in I}\Realization{Y_{\bullet}^{\alpha}}
\]
is a weak homotopy equivalence.
Then
 $f:\Realization{X_{\bullet}}\to \Realization{Y_{\bullet}}$ 
is a weak homotopy equivalence.
\begin{proof}
We do a double subdivision, and prove that 
$f:\Realization{\Sd^{2}(X_{\bullet})}\to 
\Realization{\Sd^{2}(Y_{\bullet})}$ is a weak homotopy
equivalence. 

By the assumption and by theorem~\ref{th:secondDerived}
$\{\mathrm{St}(\Sd^{2}Y_{\bullet},\Sd^{2}Y^{\alpha}_{\bullet})\}_{\alpha}$
is an open cover of $\Realization{\Sd^{2}Y_{\bullet}}$.

Let $g:A \to \Realization{\Sd{}^2Y_{\bullet}}$ 
be a map from a finite CW complex. We show that this map factors up to
homotopy over $f$. By compactness, it's image is contained in a finite union
\[
\bigcup_{i}\{\mathrm{St}(\Sd{}^2Y_{\bullet},\Sd{}^2Y^{\alpha_{i}}_{\bullet})\}.
\]   
The operation of forming the star is compatible with taking union of
simplicial subspaces, so this is the same as the subspace
\[
\mathrm{St}(\Sd{}^2Y_{\bullet},\Sd{}^2(\cup_{i}Y^{\alpha_{i}})_{\bullet}).
\]
According to theorem~\ref{th:secondDerived} 
this space is weakly homotopy equivalent to 
$\Realization{\Sd{}(\Sd{}(\cup_{i}Y^{\alpha_{i}}_{\bullet}))}$, 
so $f$ factors up to homotopy over the inclusion of this subspace.
But then, by assumption, $f$ factors up to homotopy over
$\Realization{\Sd{}(\Sd{}(\cup_{i}X^{\alpha_{i}}_{\bullet}))}$. It follows that
the map of homotopy classes
\[
f_{*}\colon [A,\Realization{X_{\bullet}}]
\to[A,\Realization{Y_{\bullet}}]. 
\]  
is surjective.
A relative argument proves that the map is injective.
\end{proof}
\end{theorem}

\subsection{A lemma in homotopy theory.}
\label{sec:an-elementary-lemma}

Suppose that $f\colon X \to Y$ is a map with the
property that point inverses are contractible. In many cases, this
implies that $f$ is a homotopy equivalence. For instance, if
$f$ is proper and $X$ and $Y$ satisfy very general conditions,
this is proved by Smale (\cite{Smale}).
However, the statement is not true without topological assumptions 
on $f$. A simple counter example is given by
$X=[0,1) \cup \{2\} \subset \RR$, $Y=[0,1]$, $f(2)=1$ and
$f(x)=x$ for $x\not=1$. We want to give a set of conditions
that ensures homotopy equivalence in some non-proper cases. 
Suppose that $X$ is a finite polyhedron (that is, the realization of
a finite simplicial set), $Y$ a
topological space and $U \subset X \times Y$ an \emph{open} set. 
Let
$\pi_{X} \colon X\times Y \to X$,
$\pi_{Y} \colon X\times Y \to Y$
 be the projections. 

\begin{lemma}
\label{le:elementary}
Assume for each $x\in X$ that 
$\pi_Y:U\cap \pi^{-1}\{x\}\to Y$  is a weak homotopy
equivalence. Then the inclusion of $U$ in $X\times Y$ is a weak
homotopy equivalence.
\end{lemma}

Before we embark on the proof, we note that in combination with
theorem~\ref{th:secondDerived} the lemma leads to the following
conclusion:

\begin{theorem}
\label{th:he-second-subdivision}
Let $X$ be a finite polyhedron and $N_{\bullet}$ a degreewise compactly
generated simplicial space. Let $K_\bullet$ be a simplicial space and
$K_\bullet \subset X \times N_{\bullet}$ a degreewise open
subspace. Let $\pi_X:K_{\bullet} \to X$ be the projection. Assume 
for each $x\in X$ that the fiber
$\vert\pi^{-1}_{X}(x)\vert\subset \vert K_\bullet\vert$ is
contractible. Then $\pi_{X}:\vert K_{\bullet}\vert \to X$ is a
homotopy equivalence.
\begin{proof}
Let $\mathrm{St}(\Sd^{2}(X\times N_{\bullet},\Sd^2 K_{\bullet})\subset X\times \vert
N_{\bullet}\vert$
be the open star in the second subdivision, and consider the diagram
\[
\xymatrix
{ 
\vert K_{\bullet}\vert\ar[r]^(.27){\simeq}\ar[dr]^{\pi_{X}}
&\mathrm{St}(\Sd^{2}(X\times N_{\bullet},\Sd^2 K_{\bullet})\ar[r]\ar[d]_{\pi_{\mathrm{St}}}
& X\times \vert N_{\bullet}\vert\ar[dl]_{\mathrm{pr}_{X}}\\
& X &
}
\]
According to theorem~\ref{th:secondDerived} the first horizontal
map is a homotopy equivalence, and since the star is open
it suffices to check that
$\pi^{-1}(x)$ is contractible. But
$\pi^{-1}_{\mathrm{St}}(x)$ is the star of the second derived
neighbourhood of $\pi^{-1}_{X}(x)$ in the simplicial space 
$\{x\}\times \lvert N_{\bullet}\rvert$, and hence contractible by the
assumption.
\end{proof}
\end{theorem}

\begin{remark}
We point out that theorem~\ref{th:he-second-subdivision} together with
theorem~\ref{th:secondDerived} proves theorem \ref{th:cuts}, since 
$N_{\bullet}=N_{\bullet}(\RR^{k})$ is contractible so that
$\lvert K_{\bullet}(X,Z)\rvert\subset \lvert X 
\times N_{\bullet}(\RR^{k})\rvert$ satisfies the
assumptions of theorem~\ref{th:he-second-subdivision}. 
\end{remark}

The remainder of this section provides a proof of
lemma~\ref{le:elementary}. Given a pair of finite polyhedra
$P\subset Q$ and a commutative diagram
\[
\xymatrix{
  P \ar@{}[r]|-{{\displaystyle\subset}}\ar[d] & Q \ar[d]^{f}\\
U  \ar@{}[r]|-{\displaystyle\subset} & X \times Y
}
\]
we must show that $f$ is homotopic to a map with image in $U$, by a
homotopy that maps $P$ to $U$ at each stage. To begin with, we may
(and will) assume that the first coordinates $f_{X}=\pi_{X}\circ
f:Q\to X$ and $f_{X}\vert P$ are simplicial maps. This follows from
simplicial approximation. We will further assume that for each simplex
$\Delta$ of $P$ we have that $f(\Delta)\subset V\times W \subset U$
for open sets $U\subset X$ and $W\subset Y$. (This might require that
we further subdivide $P$ and $Q$), Below we shall use the terminology
that a subset $A\subset U$ is neatly contained in $U$ if
$\pi_{X}(A)\times \pi_{Y}(A)\subset U$.

\emph{Step 1.} We claim that there is a subdivision of $X$ such that for
each simplex $\Delta_{\alpha}$ in the subdivision we have the following
property of the pair $Q_\alpha=f_{X}^{-1}(\Delta_{\alpha}),P_{\alpha}=Q_{\alpha}\cap P$:
There is a homotopy $H_{\alpha}:Q_{x}\times I \to X\times Y$ of
$f_{\alpha}=f\vert Q_{\alpha}$ satisfying
\begin{enumerate}[(i)]
\item 
$\pi_{X}\circ H_{\alpha}(q,t)=\pi_{X}\circ f(q)$ for $q\in Q_{\alpha}$.
\item $H_{\alpha}(p,t)=f(p)$ for $p\in P_{\alpha}$.
\item $H_{\alpha}(Q_{\alpha},1)\subset \Delta_{\alpha}\times
  W_{\alpha}\subset U$ for some open $W_{\alpha}\subset Y$.
\end{enumerate}
Note that (i) is the statement that the homotopy $H_{\alpha}$ is
effectively a homotopy of $\pi_{Y}\circ f_{\alpha}$.We now proceed to
prove the claim. For $x\in X$, let $Q_{x}=f^{-1}(\{x\}\times Y)$ and
$P_{x}=P\cap Q_{x}$. Consider the diagram
\[
\xymatrix{
P_{x} \ar[d]\ar@{}[r]|-{\displaystyle\subset}& Q_{x}\ar[d]^{f}\\
U_{x}  \ar@{}[r]|-{\displaystyle\subset} &\{x\} \times Y .\\
}
\]
By assumption, we can find a homotopy 
$H_{x}\colon Q_{x}\times I \to \{x\} \times Y$, constant on $P_{x}$,
from the restriction $f\vert_{Q_{x}}$ to a map with target inside $U_{x}$. 
We can extend the 
homotopy by a constant map on $P$, to obtain a homotopy 
from  the restriction $f \colon Q_{x} \cup P \to X\times Y$ 
to a map into $U$. By construction, this homotopy is fibrewise 
over $X$.  
Since $Q_{x}$ and $P$ are sub-polyhedra of $Q$, we can extend this homotopy
to a homotopy $H_{x} \colon Q \times I \to X \times Y$
with $\pi_{X}\circ H_{x}(q,t)=f_{X}(q)$. 

Let $h(q)=H_{x}(q,1)$. 
For every point $q \in h(Q_{x})$  
there is a neighbourhood  $V_{q}$ of $q\in X$ and an open set $W_{q}\in Y$
such that $h(q)\subset V_{q}\times W_{q} \subset U$. By compactness
we can cover $Q_{x}$ by finitely many open sets 
$h^{-1}(V_{q_{i}}\times W_{q_{i}})$. Put $V_{x}^{\prime}=\cap_{i}V_{q_{i}}$
and $W_{x}=\cup_{i} W_{q_{i}}$, so that 
$Q_x\subset h^{-1}(V_{x}^{\prime}\times W_{x})\subset U$. 
The closed set $Q \setminus h^{-1}(V_{x}^{\prime}\times W_{x})$ is compact, so
$A_{x}=\pi_{X}h(Q \setminus h^{-1}(V_{x}^{\prime}\times W_{x}))$ is a closed
set in $X$, not containing $x$. 

Let 
$V_{x}=V_{x}^{\prime}\cap(X\setminus A_{x})$. For any 
$q\in (\pi_{X}f)^{-1}(V_{x})$, we have that 
$h(q)\in V_{x}\times W_{x}\subset U$, so that
$h(\pi_{X}f)^{-1}(V_{x})\subset V_{x}\times W_{x}\subset U$. It follows that
$h(\pi_{X}f)^{-1}(V_{x})$ is neatly contained in $U$.

Using the compactness of $X$ we can find a finite covering by
such open sets $V_{x}$. After some additional subdividing each simplex
$\Delta_\alpha$ of $X$
will be contained inside one of the sets $V_{x}$. Let $H_{\alpha}$ be
the restriction of $H_{x}$ to
$Q_{\alpha}=f^{-1}_{X}(\Delta_{\alpha})$.
This completes the proof of step 1.

\vspace{1ex}
In step 1 we subdivided $X$ to obtain that for each simplex $\Delta_\alpha$
(in the subdivision) we have a homotopy 
$H_\alpha:Q_\alpha \times I \to \Delta_\alpha \times Y$
from $f\vert Q_\alpha$  to a map 
$h_\alpha:Q_\alpha \to U\cap (\Delta_\alpha \times Y)$.
We next make induction over the skeletons of $X$. For the induction step, assume
that $X$ is $n$-dimensional and let $X^{n-1}\subset X$ be the $(n-1)$-skeleton. 
The induction hypothesis is that a diagram
\[
\xymatrix{
  P_{n-1} \ar@{}[r]|-{{\displaystyle\subset}}\ar[d] & Q_{n-1} \ar[d]^{f_{n-1}}&\\
U_{n-1}  \ar@{}[r]|-{\displaystyle\subset} & X^{n-1} \times Y, 
&U_{n-1}=U\cap(X^{n-1}\times Y)\\
}
\]
with $Q_{n-1},P_{n-1}$ a pair of polyhedrons permits a homotopy
$F_{n-1}:Q_{n-1}\times I \to X^{n-1}\times Y$ from $f_{n-1}$ to a map
that sends $Q_{n-1}$ into $U$; the homotopy is relative to $P_{n-1}$ in the
sense that $F_{n-1}(p,t)=f_{n-1}(p)$ for $p\in P_{n-1}$.

Let $\{\Delta_\alpha\vert \alpha \in A\}$ be the $n$-simplices of $X$ so that
$X=X^{n-1}\cup \bigcup_{\alpha\in A}\Delta_\alpha$.

\vspace{1ex}
\emph{Step 2.} For each $n$-simplex $\Delta_\alpha$ we have the homotopy
$H_\alpha:Q_\alpha \times I \to \Delta_\alpha \times Y$ constructed in step 1. We first modify $H_\alpha$  as follows.

Choose for each $\alpha\in A$ a small affine subsimplex 
$\Delta_\alpha^0\subset \Delta_\alpha$ around the barycenter and pick a function
\[
s_\alpha:\Delta_\alpha \times I \to I
\]
with
$s_\alpha(x,t)=0$ if $x\in \partial \Delta_\alpha$, $s_\alpha(x,t)=t$ if $x\in \Delta_\alpha^0$.
Define $G_\alpha:Q\times I \to X\times Y$ to be
\[
G_\alpha(q,t)=
\begin{cases}
H_\alpha(q,s_\alpha(f_X(q),t)), & q\in Q_\alpha\\
f(q) &\text{ otherwise.}
\end{cases}
\]
where $Q_\alpha=f^{-1}(\Delta_\alpha\times Y)$. For 
$q\in Q_\alpha^0=f^{-1}(\Delta_\alpha^0\times Y)$,
$G_\alpha(q,t)=H_\alpha(q,t)$ and it follows from (iii) above that
$H_\alpha(Q_\alpha^0,1)\subset \Delta_\alpha\times W_\alpha$ for some
$W_\alpha\subset Y$. 

The homotopies $G_\alpha$, $\alpha\in A$ glue together to define 
a homotopy 
\[
G:Q\times I \to X\times Y
\]
from $f$ to $g$, $g(q)=G(q,1)$, We note the properties
\begin{enumerate}[label=(\roman*),
start=4,align=left,
]
\item $g(Q_\alpha^0)\subset \Delta_\alpha \times W_\alpha\subset U$, $\alpha\in A$.
\item  $G(p,t)=f(p)$, $p\in P$.
\item $\pi_X\circ G(q,t)=f(q)$, $q\in Q$.
\end{enumerate}

\emph{Step 3.} For each $n$-simplex $\Delta_\alpha$, let 
$u_\alpha: \Delta_\alpha \times I \to \Delta_\alpha$ be a homotopy such that.
\[
u_\alpha(x,0)=x, u_\alpha(x,t)=x\text{ if }
x\in \partial \Delta_\alpha,u_\alpha(x,1)\in\partial \Delta_\alpha
\text{ if } x\in \Delta_\alpha\setminus \Delta_\alpha^0 
\]
We use $u_\alpha$ to define a homotopy $K:Q\times I\to X\times Y$
of the map $g:Q\to X\times Y$ of step 2:
\begin{align*}
K(q,t) &= g(q) \text{ if } q\in Q\setminus \cup_{\alpha \in A}
\inter \Delta_\alpha \times Y \\
\pi_XK(q,t) &=u_\alpha(\pi_X\circ q(q),t), q\in g^{-1}(\Delta_\alpha\times Y)\\
\pi_YK(q,t) &=\pi_Yg(q), q\in g^{-1}(\Delta_\alpha\times Y)\\
\end{align*}

The new map $k(q)=K(q,1)$ maps $\bigcup_{\alpha} g^{-1}(\Delta_\alpha^0\times Y)$
into $U$ and $\pi_X\circ k$ maps $Q\setminus \bigcup_\alpha \Delta_\alpha^0$ into 
$X^{n-1}$.

Set $Q_1=Q\setminus k^{-1}(\cup_{\alpha \in A}\inter \Delta_\alpha^0\times Y)$,
and
$P_1=P\cap Q_1 \cup \cup_\alpha k^{-1}(\partial \Delta_\alpha^0)$.

Then we have the diagram
\[
\xymatrix{
  P_1 \ar@{}[r]|-{{\displaystyle\subset}}\ar[d] & Q_{1} \ar[d]^{k}\\
U \cap(X^{n-1}\times Y)  \ar@{}[r]|-{\displaystyle\subset} & X^{n-1} \times Y.\\
}
\]
By the inductive assumption we can find a homotopy of $k$ to a map 
$Q_1\to U\cap (X^{n-1}\times Y))$ such that the homotopy is constant on $P_1\times I$ and hence extends to all of $Q$. This completes the proof.

\section{Metrizability of  $\Psi_d(\RR^{d+n})$}
\label{sec:topology}
This section proves theorem~\ref{le:topology}, which states that
$\Psi_{d}(\RR^{d+n})$ is metrizable in the topology defined in \S~
\ref{sec:critical-pairs-cuts}. In \ref{se:Redefinition} we give an
equivalent definition of the topology and show that
$Psi_{d}(\RR^{d+n})$
is a regular space. In the following \S~\ref{se:countable topology}
we show that the topology is also countable, and hence by a standard 
theorem that $Psi_{d}^{d+n}$ is metrizable. 
\subsection{Redefinition of the topology}
\label{se:Redefinition}
Recall that 
\[
\Psi_{d}(\RR^{d+n})=
\{W^{d}\subset \RR^{d+n}\mid \partial W=\emptyset, \text{ $W$ a closed subset}\}.
\]
Suppose that $W\in \Psi^{d}(\RR^{d+n})$. To a point $x\in W$ we
associate the point in $G(d,n)$, the Grassmannian of $d$-planes in
$\RR^{d+n}$. The point is  determined by the tangent space 
of $W$ at $x$. This gives a map
\[
W \to \RR^{d+n}\times \mathrm{G}(d,n).
\]
Chose a metric $\mu$ on the compact manifold $\mathrm{G}(d,n)$, 
and consider the product metric $\mu_{2}$ on  
$\RR^{d+n}\times \mathrm{G}(d,n)$.

For every positive $r\in \RR$, let $rD^{d+n}$ be the closed disc around
the origin of radius $r$.
Let $V,W\in \Psi_{d}(\RR^{d+n})$. 
We want to say that $V$ and $W$ are close if there is a diffeomorphism $\phi$ 
between $V\cap rD^{n+s}$ and $W\cap rD^{n+s}$, such that 
$p$ is close to $\phi(p)$ and $T_p(W)$ is close to $T_{\phi(p)}(V)$ in the chosen
metric on the Grassmannian space. 
But since a diffeomorphism doesn't necessarily preserve the distance to the origin, 
we are going to use a more careful formulation.    

Instead, we consider diffeomorphisms
$\phi:U\to s(U)$ where $U\subset V,\phi(U)\subset W$ are open subsets. 
Let ${\cal P}_{r}(V,W)$ be the set of such 
diffeomorphisms $\phi$ that satisfy
$V\cap rD^{d+n}\subset U$ and
$W\cap rD^{d+n}\subset \phi(U)$. 
When we measure the distance between $p$ and $\phi(p)$, we will only care about
points $p \in s(\phi,r)=(V\cap rD^{d+n})\cup(\phi^{-1}(W\cap rD^{d+n}))$.
So, for a diffeomorphism $\phi\in{\cal P}_r(M,N)$ we define 
$d_r^\prime(\phi)=\sup_{p\in s(\phi,r)}\mu_{2}(p,\phi(p))$. 
If $s(\phi)=\emptyset$, this number is understood to be 0.
Finally,  we define
\begin{equation}
\label{eq:DefDistance}
d_{r}(V,W)=
\inf_{\phi\in{\cal P}_{r}}(d_r^\prime(\phi))\in \RR_{+}\cup\{\infty\}.
\end{equation}
Again, if ${\cal P}_r(V,W)=\emptyset$, this number is understood to be $\infty$.

\begin{remark}
\label{re:germproperty}
The notation above is slightly abusive, since
the definition of $d^{\prime}_{r}(\phi)$ actually involves the 
source $U$ of $\phi$. This dependence is only weak. 
In the definition of $d^\prime$ we only use the values of $\phi$ in the 
closed subset $s(\phi,r)$. This
means that if $V$ is an arbitrary small neighbourhood  
of $s(\phi,r)$ in $U$, then $d^{\prime}_{r}(\phi\mid_V)=d^{\prime}_{r}(\phi)$. 

For example, suppose that $d_{r}(V,W)<d$. We can find $\phi:U\to \phi(U)$ such that
$d_r^\prime(\phi)<d$. As above, by restricting $\phi$ to a neighbourhood of
$s(\phi,r)$ in $U$, we can assume that $U$ satisfies that
$\mu_{2}(p,\phi(p))<d$ for all $p\in U$. 
\end{remark}

Here are the main properties of the functions $d_r$:

\begin{lemma}
\label{le:pseudometric}
Let $W_{i}\in \Psi_{d}(\RR^{d+n})$. Then
\begin{enumerate}[(i)]
\item
\label{enum:symmetry}
 Symmetry: $d_{r}(W_{1},W_{2})=d_{r}(W_{2},W_{1})$.
\item
\label{enum:semicontinuity}
Semi-continuity: If $r\leq r^\prime$, then $d_r(W_1,W_2)\leq d_{r^\prime}(W_1,W_2)$.
If $d_{r}(W_{1},W_{2})<\epsilon$, there are
$\delta,\epsilon_{1}>0$ so that 
$d_{r+\delta}(W_{1},W_{2})<\epsilon-\epsilon_{1}$. 
\item 
\label{enum:triangle}
Weak triangle inequality: Given numbers $r_{23}>r_{12}>r_{13}>0$, the triangle 
inequality 
\[
d_{r_{13}}(W_1,W_3)\leq d_{r_{12}}(W_1,W_2)+d_{r_{23}}(W_2,W_3)
\]
is valid for triples of manifolds $\{W_1,W_2,W_3\}$ that satisfy the
additional assumption that
$d_{r_{12}}(W_{1},W_{2}) < r_{23}-r_{12}$ and
$d_{r_{23}}(W_{2},W_{3}) < r_{12}-r_{13}$.
\end{enumerate}
\end{lemma}
The semi-continuity property says that if we fix $W_1$ and $W_2$,
the function $f(r)=d_r(W_1,W_2)$ is a monotonously increasing
upper semi-continuous function of $r$. It is not continuous in general.

The functions $d_r$ do not satisfy the usual triangle inequality. But they
do satisfy the weak form of the triangle inequality above, involving
several $r$. There is the technical difficulty that this 
triangle inequality is only true if the manifolds are close to each
other, that is if they satisfy the additional assumption of \ref{le:pseudometric}.\ref{enum:triangle}. 
Fortunately it turns out that these restrictions are not so
important when you study the topology defined by the all the functions
$d_r$.  
 
\begin{proof}[Prof of lemma \ref{le:pseudometric}]
(\ref{enum:symmetry}) The definition of $d_{r}$ is symmetric
in $W_{1}$ and $W_{2}$,
exchanging $\phi$ for $\phi^{-1}$.

(\ref{enum:semicontinuity}) 
Choose $\epsilon_{1}>0$ so that 
$d_{r}(W_{1},W_{2})<\epsilon-\epsilon_{1}$.
Let $\phi:U \to \phi(U)$ be in ${\cal P}_{r}(W_{1},W_{2})$ such that
for $p \in (W_1\cap rD^{n+s})\cup \phi^{-1}(W_2\cap rD^{n+s})$ we have that
$\mu_{2}(p,\phi(p))< \epsilon-\epsilon_{1}$. 

Since $U\subset W_{1}$ is open, it follows that the set
$\{p\in U\vert \mu_{2}(p,\phi(p))<\epsilon-\epsilon_{1}\}$ 
is open in $W_1$. That is, for some $\delta>0$ it contains
$(W_1\cap (r+\delta)D^{n+s})\cup \phi^{-1}(W_2\cap (r+\delta)D^{n+s})$.
But then $\phi\in {\cal P}_{r+\delta}(W_{1},W_{2})$, and the statement follows.

(\ref{enum:triangle})
Consider three manifolds $W_{1},W_{2},W_{3}$ which satisfy the
conditions $r_{23}-r_{12}> d_{r_{12}}(W_{1},W_{2})$ and
$r_{12}-r_{13}>d_{r_{23}}(W_{2},W_{3})$. 
Let $d_{12}$ and $d_{23}$ be real numbers such that  
$r_{23}-r_{12}>d_{12}> d_{r_{12}}(W_{1},W_{2})$ and
$r_{12}-r_{13}>d_{23}> d_{r_{23}}(W_{2},W_{3})$. 

In order to prove (3), we will show that 
$d_{12}+d_{23} \geq d_{r_{13}}(W_{1},W_{3})$.
Using remark \ref{re:germproperty} we can find   
$\phi_{12}: U_{1}\to \phi_{12}(U_{1})$ in ${\cal P}_{r_{12}}(W_{1},W_{2})$ and
$\phi_{23}:U_{2}\to \phi_{23}(U_{2})$ in ${\cal P}_{r_{23}}(W_{2},W_{3})$ 
such that for all $p\in U_3$ respectively for all $q\in
U_2$ we have that $\mu_{2}(p,\phi_{23}(p))<d_{23}$ and
$\mu_{2}(q,\phi_{12}(q))<d_{12}$.

In order to bound $d_{r_{13}}(W_{1},W_{3})$ from above, we 
need to construct an element in ${\cal P}_{r_{13}}(W_{1},W_{3})$. The obvious
choice to try is the composition $\phi_{23}\circ \phi_{12}$. 
One problem is that this composition might not be defined on all of $U_1$. 
The first step of the argument is  to check that the composition is
defined at least on an open neighbourhood of $s(\phi_{12},r_{12})$ in $U_{1}$.

What makes this work is that for $p\in U_{1}$ we are assuming that
$\mu_{2}(p,\phi_{12}(p))<r_{23}-r_{12}$. Because it follows from this 
and the triangle inequality for $\mu_{2}$ that
if $p\in U_{1} \cap r_{12}D^{n+s}$, then $\phi_{23}(p)\in
r_{23}D^{n+s}$.  

In particular, since $W_{2}\cap r_{23}D^{n+s}\subset
U_{2}$ we conclude that possibly after replacing $U_{1}$ by a smaller open
subset containing $s(\phi_{23},r_{1})$, we can assume that
$\phi_{12}\circ \phi_{23}$ is defined on $U_{1}$.

We now claim that 
$\phi_{23}\circ \phi_{12}\in {\cal P}_{r_{13}}(W_{1},W_{3})$. What
we have to prove is  that
 $W_{3}\cap r_{13}D^{s+n}\subset \phi_{23}\circ
\phi_{12}(U_{1})$. But if 
$q\in W_{3}\cap r_{13}D^{s+n}\subset W_{3}\cap r_{12}D^{s+n}$
then $q=\phi_{23}(p)$ for some $p\in U_{2}$. Since
$\mu_{2}(p,\phi_{23}(p))<r_{12}-r_{13}$, we have that $p\in W_{2}\cap
r_{12}D^{s+n}\subset \phi_{12}(U_{1})$, so
$q=\phi_{23}(p)\in \phi_{23}\circ \phi_{12}(U_{1})$.

Finally, the triangle inequality for $\mu_{2}$ shows that
if $p\in s(\phi_{23}\circ\phi_{12},r_{3})$, 
then 
\[
\mu_{2}(p,\phi_{23}\circ\phi_{12}(p))\leq 
\mu_{2}(p,\phi_{12}(p))+
\mu_{2}(\phi_{12}(p),\phi_{23}\circ\phi_{12}(p))\leq d_{12}+d_{23}.
\]
 
This proves (\ref{enum:triangle}) and completes the proof. 
\end{proof}
Given $M\in \Psi_{d}(\RR^{d+n}_{d})$ we define the neighborhoods
\[{\cal U}_{r,\epsilon}(M)=
\{N\in \Psi^{d}(\RR^{d+n})
\mid
d_{r}(M,N)< \epsilon\}.
\]

\begin{lemma}
\label{le:DefTopology}
The sets ${\cal U}_{r,\epsilon}(M)$ form a basis for 
a topology on $\Psi^{d}(\RR^{d+n})$. This topology is regular.
\begin{proof}
To show that the sets ${\cal U}_{r,\epsilon}(M)$
form the basis of a topology, we need to show that if
$N\in {\cal U}_{r_{1},\epsilon_{1}}(M_{1})
\cap {\cal U}_{r_{2},\epsilon_{2}}(M_{2})$, then there exist
$\epsilon^{\prime},r>0$ such that 
\[
{\cal U}_{r^{\prime},\epsilon^{\prime}}(M,)\subset
{\cal U}_{r_{1},\epsilon_{1}}(M_{1})
\cap {\cal U}_{r_{2},\epsilon_{2}}(M_{2}).
\]
It is enough to show that if  
$N\in {\cal U}_{r,\epsilon}(M)$, then there exists $r^{\prime}$ and
$\epsilon^{\prime}$ such that 
$
{\cal U}_{r^{\prime},\epsilon^{\prime}}(N)\subset{\cal U}_{r,\epsilon}(M)
$.

By lemma~\ref{le:pseudometric}.\ref{enum:semicontinuity} 
there are $\delta,\epsilon_1>0$ so that 
$N\in {\cal U}_{r+\delta,\epsilon-\epsilon_1}(M)$.
Choose a positive $\epsilon^{\prime}<\min(\delta,\epsilon_{1})$,
and put $r^\prime=r+\delta+\epsilon-\epsilon_{1}$. 
We claim that 
${\cal U}_{r^{\prime},\epsilon^{\prime}}(N)\subset {\cal
  U}_{r,\epsilon}(M)$.
Let $N^{\prime}\in {\cal  U}_{r,\epsilon}(M)$
Use lemma~\ref{le:pseudometric}.\ref{enum:triangle}  with $W_{1}=M$, $W_{2}=N$,
$W_{3}=N^{\prime}$, $r_{12}=r+\delta$, $r_{23}=r^{\prime}$ and $r_{13}=r$.
The conclusion is that 
$d_{r}(M,N^{\prime})< \epsilon$, which proves our claim.

To prove regularity, it suffices to show that for every $r,\epsilon>0$
we  separate $M$ from the complement of ${\cal U}_{r,\epsilon}(M)$.
To do this, it suffices to find an
open neighbourhood $U$ of $M$, and for every $N$ in
the complement of ${\cal U}_{r,\epsilon}(M)$ an open neighbourhood 
$V_{N}$ of $N$ disjoint from $U$.

We chose $U={\cal U}_{r+1,\epsilon/2}(M)$.
Let 
$r^{\prime}=r+1+\epsilon$ and $\epsilon^{\prime}=\min(\epsilon/2,1)$.
For $N\not\in {\cal U}_{r,\epsilon}(M)$, we chose
$V_{N}={\cal U}_{r^{\prime},\epsilon^{\prime}}(N)$.
We need to show that $U\cap V_{N}=\emptyset$. So assume to the
contrary that $N^{\prime}\in U\cap V_{N}$. 
From lemma~\ref{le:pseudometric}.\ref{enum:triangle} with
$W_{1}=M$, $W_{2}=N^{\prime}$, $W_3=N$,
$r_{12}=r+1$, $r_{23}=r+1+\epsilon/2$
and $r_{13}=r$ we obtain that $d_{r}(M,N)<\epsilon$, in contradiction 
to the assumption on $N$.
\end{proof}
\end{lemma}

\subsection{Countability of the topology}
\label{se:CountableTopology}
\begin{theorem}
\label{th:CountableTopology}
The topology of $\Psi_{d}(\RR^{n+s})$ has a countable basis.
\end{theorem}
We remind the reader that the elements 
$M\in \Psi_{d}(\RR^{n+s})$ are smooth $d$-dimensional 
submanifolds which are closed sets
in $\RR^{d+n}$ . We will need to introduce a list of curvature conditions 
of $M\subset \RR^{d+n}$. Loosely speaking, these conditions will 
bound a measure of curvature from above by $K$,
but only on a normal
tube around $M$ of radius $\delta$ and inside a disc $rD^{d+n}$. 
The set of manifolds
satisfying the conditions with respect to the numbers $r,\delta,K>0$
form a set ${\cal X}_{{r,\delta}}(K)\subset \Psi_{d}(\RR^{d+n})$. We defer the 
precise formulation of the curvature conditions to later. Instead we 
formulate the properties of the sets ${\cal X}_{r,\delta}(K)$ as lemma
\ref{le:ClassesOfManifolds}.  
Then we reduce the proof of theorem \ref{th:CountableTopology} to lemma
\ref{le:ClassesOfManifolds}, and finally we discuss the proof of the lemma.

For $M,N$ in $\Psi_{d}^{d+n}$, we define the $r$-Hausdorff distance as
\[
d^{H}_{r}(M,N)=\max(\sup_{x\in M\cap rD^{n+s}}d(x,N),\sup_{x\in N\cap
  rD^{n+s}}(d(x,M)).
\]
 If $d^{H}_{r}(M,N)$ is small, the two manifolds are
pointwise close to each other after intersecting with  $rD^{n+s}$, but we don't
assume that the tangent spaces at close points are close. 
In general, a bound on the Hausdorff distance between $M$ and $N$ does
not give a bound on the distance $d_{r}(M,N)$ defined by
(\ref{eq:DefDistance})
in \S \ref{se:CountableTopology}.  

\begin{lemma}
\label{le:ClassesOfManifolds}
For $\delta,r,K>0$ 
there is a subset ${\cal
  X}_{r,\delta}(K)\subset \Psi_{d}^{d+n}$  with the following properties.
\begin{enumerate}[(i)]
\item 
\label{le:ClassesOfManifolds:ex}
Given $M\in \Psi_{d}^{d+n}$ and any 
$r>0$ there are $\delta>0,K>0$ such that
 $M\in {\cal X}_{r,\delta}(K)$.
\item 
\label{le:ClassesOfManifolds:mono}
If $r^{\prime} \leq r, \delta^{\prime} \leq \delta, K^{\prime} \leq K$,
then  ${\cal X}_{r,\delta}(K^{\prime}) \subset {\cal X}_{r^\prime,\delta^{\prime}}(K)$.
\item  
\label{le:ClassesOfManifolds:sep}
For any $r,\epsilon,K>0$ there exists $\delta>0$ 
with the following property: If
$\delta^\prime <\delta$, $M,N\in {\cal
  X}_{r+1,\delta^\prime}(K)$ and 
$d_{r+1}^{H}(M,N)<\delta^\prime$, then  $d_r(M,N) < \epsilon$. 
\end{enumerate}
\end{lemma}

We want to construct a countable, dense set in $\Psi_d(\RR^{d+n})$.
This set will depend on choices to be specified below. 

For the moment, we fix numbers  $r,\delta>0$. 
Chose a finite set of points $\{x_{i}\}_{i\in I}\subset rD^{d+n}$, $i\in I$  
such that for any $x\in rD^{d+n}$ there is an $i$ such that
$d(x,x_{i})<\delta/2$. For any $N\in \Psi_{d}^{d+n}$ we define 
$S(N)=\{i\in I\vert d(x_{i},N)<\delta/2\}\subset I$. If $S(N_1)=S(N_2)$,
then obviously $d^{H}_{r}(N_1,N_2)<\delta$. 

There will be a finite family $J$ of sets $j\subset I$ such that there exists
an $N\in {\cal X}_{r,\delta}(K)$ with $j=S(N)$. For each $j\in J$ chose
a manifold $N_{j}$ such that $j=S(N_{j})$.  

\begin{definition}
Let ${\cal N}_{r,\delta}(K)$ be this set of manifolds.
\end{definition}

\begin{remark}
\label{Hdense}
For any $M\in {\cal X}_{r,\delta}(K)$, there is an $N\in {\cal N}_{r,\delta}(K)$ such
that the Hausdorff distance satisfies $d^H_r(M,N)<\delta$. Actually, we can chose
$N$ as the unique $N_j\in {\cal N}_{r,\delta}(K)$ such that  $S(N_j)=S(M)$.
\end{remark}

\begin{lemma}
\label{le:countable}
The subset 
\[
{\cal N}=\bigcup_{r,\delta,K\in \QQ_+}{\cal N}_{r,\delta}(K)\subset \Psi_d(\RR^{d+n})
\]
is countable and dense.
\begin{proof}
${\cal N}$ is countable since it is a countable union of finite sets. 
We have to prove that it is dense, that is, for any
$M\in \Psi^r(\RR^{d+n})$ and any $r,\epsilon>0$ there is an $N\in\cal N$ such that
$d_r(M,N)<\epsilon$. 

By \ref{le:ClassesOfManifolds}.\ref{le:ClassesOfManifolds:ex}
we have that $M\in{\cal X}_{r+1,\delta_0}(K)$ for some
rational $r>0$ and $\delta_0,K>0$. By
 \ref{le:ClassesOfManifolds}.\ref{le:ClassesOfManifolds:mono}
we can assume that  $K$ is also a rational number.
Using \ref{le:ClassesOfManifolds}.\ref{le:ClassesOfManifolds:sep} we find
$\delta_1>0$ such that if $M,N\in {\cal X}_{r,\delta}(K)$
and $0<\delta<\delta_1$,  then $d_r(M,N)<\epsilon$.

Pick $\delta>0$ to be a rational number, $\delta<\min(\delta_0,\delta_1)$.
Then $M\in {\cal X}_{r+1,\delta}(K)$  
by \ref{le:ClassesOfManifolds}.\ref{le:ClassesOfManifolds:mono}.
According to remark \ref{Hdense} there is an $N\in {\cal N}_{r,\delta}(K)$
such that $d_r^H(M,N)<\delta$. Since $\delta<\delta_1$, it follows that
$d_r(M,N)<\epsilon$, completing the proof of the lemma.
\end{proof}
\end{lemma}

\begin{proof}[Proof of theorem \ref{th:CountableTopology}]
We claim that the open sets ${\cal U}_{r,\epsilon}(N)$
with $N\in {\cal N}$ and  rational $r,\epsilon$ form a countable basis
for the topology.  To prove that this is a basis,  it suffices to show that 
for any $M,M^{\prime},r_1,a_1$ such that 
$M\in {\cal U}_{r,d}(M^\prime)$, there is an $N\in \cal N$ 
and $r,\epsilon>0$
so that
$M\in {\cal U}_{r,\epsilon}(N)
\subset {\cal  U}_{r_1,a_1}(M^\prime)$. By 
lemma~\ref{le:pseudometric}.\ref{enum:semicontinuity},
we can find $r_2 >r_1$ and a positive 
$a_2<a_1$ such that
$d_{r_2}(M,M^\prime)<a_2$.

Using
lemma~\ref{le:pseudometric}.\ref{enum:triangle} we can find $a,b>0$ such that
if $N\in {\cal U}_{r_2+b,a}(M)$, then
$M\in {\cal U}_{r_2+b,a}(N)\subset {\cal U}_{r_1,a_1}(M^\prime)$.
For instance, $a=(a_1-a_2)/2$ and $b=\max((r_2-r_1)/2,a_1)$ will do.
Since $\cal N$ is dense, we can find such an $N\in\cal N$, which concludes the
proof of the theorem. 
\end{proof}

We now need to define the class $X_{r,\delta}(K)$. To define this class, we write down a sequence of conditions on a manifold $M$. These conditions depend on the positive numbers $r,\delta$ and $K$. For any manifold embedded in $\RR^{n+s}$, we consider the exponential map
\[
E_M:\nu_\delta(M) \to \RR^{n+s} \qquad E_M(p,v)=v+p.
\]

{\bf Condition $1$.}
If we restrict $E_{M}$ to $\{(p,v)\in \nu M; \norm p <r,\norm
v<\delta\}$ 
this map becomes a diffeomorphism onto its image. 

{\bf Condition $2$.} The normal curvature of $M$ is bounded by $K$ in 
$M\cap rD^{n+s}$. 

The inverse of $E_{M}$ followed by projection to $M$ is a differentiable map
\[
F: E_M(\nu_\delta) \to M.
\]
The last conditions is concerned with this map. We consider $F$
as a map from an open subset of $\RR^{n+s}$ to $\RR^{n+s}$. In
particular, we can define arbitrary partial derivatives of $F$.
For any given manifold $M$, these derivatives are bounded on any
compact subset.

{\bf Condition $3$.} All first and second order derivatives
of $F$ are bounded by $K$, that is
\[
\Big\vert\frac{\partial F_i }{\partial x_j}(x)\Big\vert <K, \qquad
\Big\vert\frac{\partial^2 F_i }{\partial x_jx_k}(x)\Big\vert < K \quad\quad
\text{ for } x \in E_M(\nu_\delta\vert_{M\cap RD^{n+s}})
\]

The final condition is topological. 

{\bf Condition $4$.} If $\epsilon < 4\delta$ and $p \in M$, 
then $M\cap \inter D_\epsilon(p)$ is contractible. 

\begin{definition}
$X_{R,\delta}(K)$ is the set of manifolds $M\in\psi_d(\RR^{d+n})$
 satisfying the above conditions 1--4. 
\end{definition}

We now turn to
\begin{proof}[Proof of lemma \ref{le:ClassesOfManifolds}]
The first two statements of the lemma are easy to verify,
but we do have to prove the third statement. 
So let $r,\epsilon,K>0$ be given. We need to specify a $\delta>0$. 

Suppose that $M\in X_{r+1,\delta}(K)$ and that $N$ has normal
curvature less than $K$. Further suppose that $N\cap (r+1)D^n$ is
contained in an $\delta$ tubular neighbourhood of $M$. This defines a projection map $\pi:N \to M$. If $\delta<\epsilon/2$ then $d(x,\pi(x))<\epsilon/2$. 

We want to show that we can choose $\delta$ so that the distance in the Grassmannian manifold between $T_p(N)$ and $T_{\pi(x)}(M)$ arbitrary smaller than $\epsilon/2$  for $p\in N\cap rD^{n+s}$. 

Let $p\in N\cap rD^{n+s}$ and  $v\in T_p(N)$ a unit length tangent vector. Pick a geodesic $\gamma$ in $N$ so that $\gamma(0)=p$ and $\gamma^\prime(0)=v$. The geodesic is defined for all $t$, and if $\norm t <1$ we have that $\gamma(t)$ is contained in an $\delta$ tubular neighbourhood of $M$.
Let $f(t) = \gamma(t)-F\gamma(t)$. By elementary calculation, the second derivative of the components of $f$ satisfies
\[
f^{\prime\prime}(t)_i=\gamma_i^{\prime\prime} 
-\sum_{j,k}\frac{\partial^2F_i}{\partial x_j\partial x_k}\gamma^\prime_j\gamma^\prime_k
-\sum_j\frac{\partial F_i}{\partial x_j}\gamma^{\prime\prime}_j
\]
Recall that $\gamma$ is parametrized by arc length and has curvature bounded by $K$. By condition 3, $\norm{f^{\prime\prime}(t)_i}\leq K_1$ where
$K_1=K+K^2+K^3$.

% Then, $f^\prime(t)=\gamma^\prime(t)-DP_{\gamma(t)}(\gamma^\prime(t))=\gamma^\prime(t)-\gamma^\prime(t)(P)$,
% and $f^{\prime\prime}(t)=\gamma^{\prime\prime}(t)-H(P)(\gamma^\prime(t),\gamma^\prime(t))$. 

So we obtain:  
\begin{equation}
\label{eq:Landau}
\begin{gathered}
\norm {f(t)_i} < \delta\\ 
\norm{f^{\prime\prime}(t)_i} < K_1.
\end{gathered}
\end{equation}
By elementary arguments, if a twice differentiable function defined on $[-1,1]$ satisfies (\ref{eq:Landau}), it follows that $\norm {f^\prime(0)} < \max(2\delta,2\sqrt{\delta K_1})$. This is a special case of the Landau-Kolmogorov  inequalities. Given $K$ and thus $K_1$, by choosing $\delta$ small enough, we can assure that $\norm{f^\prime(0)}=\norm{v-(F\gamma)^\prime(0)}$ can be made arbitrarily small. 

But since $(F\gamma)^\prime(0)$ is contained in the tangent space of
$M$ at $F(p)$, it follows that the projection of $v$ to the normal
space of $M$ at $F(p)$ can also be made arbitrarily small. 

Since $v$ is a unit vector, we see that we can chose $\delta$ so small
that the distance between $T_pN$ and $T_{P(p)}(M)$ can also be made smaller than $\epsilon/2$ in the Grassmannian space.

It follows that the map $F_{M}:N\to M$ satisfies that
$d_r(x,F(x))<\epsilon$.
To finish the proof, we have to show that 
$F$ is a diffeomorphism on its image.
  
At least we know that the map $F_{M}:N \to M$ is a local diffeomorphism,
since its differential is the projection of $T_pM$ to $T_{F(p)}N$. 
Let $U= N\cap (r+1-\delta)\inter D^{n+s}$. We have to show that
$F_M$ is injective on $U$, and that its image is an open set
in $M$ which contains $M\cap rD^{n+s}$. 

It's easy to see that $F_M(U)\cap (r+\delta)D^{n+s}$ is open and closed
in $M \cap (r+\delta)D^{n+s}$, so the image consists of a union of components.
If $p\in M\cap rD^{n+s}$, there is a point $q\in N$ such that
$d(p,q)<\delta$ and  $d(F_M(q),p)<2\delta$. By condition 4,
$p$ and $F_M(q)$ are in the same component of $M\cap (r+\delta)D^{n+s}$.
It follows that $p\in F_M(U)$.

Finally we need to prove that $F_M$ is injective on $U$. Suppose 
$p,q\in U$, and $F_M(p)=F_M(q)$. There is a curve connecting $p$ and $q$
of diameter less than $2\delta$. Its image in $M$ is a closed curve
of diameter less than $4\delta$, so it is null-homotopic. Since
$F_M$ is a covering map, $p=q$.
\end{proof}

\begin{lemma}
The topology we have defined on $\Psi^d(\RR^{d+n})$ has a second
countable basis. Moreover, $\Psi^d(\RR^{d+n})$ is a separable, 
metrizable topological space.
\begin{proof}
We prove the following statement:
For every $\epsilon, r, C >0$
there is a countable subset $X\subset  \Psi(\RR^{d+n})$ such that
for every $W\in \Psi(\RR^{d+n})$ whose normal curvature is bounded 
by $C$  there is an $M\in X$ such that
$d_r(W,M)<\epsilon$.

Given this, we can find a countable set $X$ such that for every 
$r, \epsilon>0$ and $W$ we can find an $M\in X$ such that
$d_r(M,W)<\epsilon$. This set is clearly dense, and using lemma
~\ref{le:pseudometric} we see that the sets
$\{{\cal U}_{r,\epsilon}(M)\mid M\in X;r,\epsilon\in  \mathbb Q\}$ form 
a countable basis for the topology. It follows from
Uhrysohns metrication theorem that the topology is metrizable.
By lemma~\ref{le:countable} it is separable.
\end{proof} 
\end{lemma}

\subsection{Equivalence of topologies.}

We need to compare the topology of lemma \ref{le:DefTopology} 
to the tolology defined in section~\ref{sec:critical-pairs-cuts}. 

For $M\subset \RR^{d+n}$ the fine ${\cal C}^1$ topology on the space of 
${\cal C}^1$ maps $f:N\to M$ as the topology generated by the sets
\[
f:N\to M;\quad \norm{f(x)-g(x)}<\delta(x), \norm{df_x(v)-dg_x(v)}<\delta(x)\norm{v}
\]
where $\delta: M\to \RR$ is a positive, continuos function.
See \cite{Munkres}, definition 3.5.

We will need: 

\begin{theorem}[\cite{Munkres}, theorem 3.10]
\label{embeddings}
Let $M\to N$ be a ${\cal C}^{1}$ map.
If $f$ is a diffeomorphism, there is a fine neighbourhood of $f$ such
that if $g$ is in this neighbourhood, then $g$ is a diffeomorphism.
\end{theorem}

We will  need the following elementary estimate.

\begin{lemma}
\label{le:elementary}
Let $V\subset \RR^{d+n}$. Let $A:V\to V,A^{\perp}:V\to V^{\perp}$ be
linear
maps. Let $W=(\mathrm{Id}+ A,A^{\perp})V\subset \RR^{d+n}$. For any $\epsilon>0$
there
are numbers $\delta_{1}>0$, $\delta_{2}>0$ such that if
$\norm{A}<\delta_{1}$ and in the metric of the Grassmannian 
$d(V,W)<\delta_{2}$, then $\norm{(A,A^{\perp})}<\epsilon$.\hfill$\square$
\end{lemma}

\begin{proof}[Proof of lemma~\ref{le:topology}] 
We need to show that the topology $\cal T$ defined by the sets ${\cal N}_{K,\epsilon'}(W)$
agrees with the  topology on  $\Psi^d(\RR^{d+n})$
we have defined above.

We first show that $\cal T$ is finer than the topology of
 $\Psi^d(\RR^{d+n})$.

For any  $W,r,\epsilon$ 
we can find $K,\epsilon'$ so that
$W\in {\cal N}_{K,\epsilon'}(W) \subset {\cal U}_{r,\epsilon}(W)$.
If $s:M\to \RR^{d+n}$ is defined by
a small section of the normal bundle we put (as above)  $\phi(x)=x+s(x)$.
Then $\phi$ is a diffeomorphism from $M$ to $W$.  A vector
$v\in T_xM\subset \RR^{d+n}$ is close to the corresponding
vector $v+ds_x(v)\in T_{\phi(x)}(s(M))\subset \RR^{d+n}$. It follows that the vectorspace
$T_{s(x)}(s(M))$ is close to $T_x(M)$ in the Grassmannian, and that we
can choose $\epsilon$ so small that 
${\cal N}_{K,\epsilon'}(W) \subset {\cal U}_{r,\epsilon}(W)$.

In the general case, if $M\in {\cal U}_{r,\epsilon}(W)$, we can
use 
lemma~\ref{le:pseudometric}.\ref{enum:semicontinuity}
to find $r',\epsilon'$ so that 
$M\in {\cal U}_{r',\epsilon'}(M) \subset {\cal U}_{r,\epsilon}(W)$,
and apply the above argument to $M\in {\cal U}_{r',\epsilon'}(M)$.
It follows that the topology defined by the ${\cal N}_{K,\epsilon'}(W)$ 
is at least as fine as the topology on $\Psi^d(\RR^{d+n})$.

We have to prove the opposite implication.
Given  $W,K,\epsilon^\prime$
we need to find $r,\epsilon$ so that
\[
W \in {\cal U}_{r,\epsilon}(W) \subset {\cal N}_{K,\epsilon'}(W).
\]

If $\epsilon$ is
sufficiently small (depending on $W$ and a number $r>0$), the map 
\[
e : \nu (W \cap r\mathrm{int} D^{d+n}) \to \RR^{d+n},\quad (x,v) \mapsto x+v
\]
is a diffeomorphism onto its image.

It's inverse followed by the projection
defines a differentiable map $\pi$, where $\pi(x)$ 
denotes the unique point on $W$
which is closest to $x$. If $\epsilon$ is sufficiently small, the
composite  $\pi \circ s$ will be close to the identity in the
${\cal C}^{1}$ topology. 

It follows from theorem \ref{embeddings}
that $\pi \circ s$ is a diffeomorphism onto its image.
The inverse of $\pi$ is given by a section $s$ in the normal
bundle of $M$, and $W=\pi^{-1}(M)$ 
We still need to show that possibly after decreasing $\epsilon$  
we can make the norm of $s$ arbitrarily small. 

For each $x\in M$, we can write the differential $ds$ as a sum
$ds_{\tau}\oplus ds_{\nu}$ where 
$ds_{x\tau }\in T_{x}(M)$ and
$ds_{x\nu}\in \nu_{x}(M)$.

Supppose that the tangent bundle of $M$ has a family of
sections $t_{i}$, forming an orthogonal basis at each point.
Since $<s,t_{i}>=0$ we have that for any section $s$,
\[
\norm{ds_{x\tau}(v)}^{2}
=
(\sum_{i}{<ds(v),t_{i}>})^{2}
\leq 
\norm{s}^{2}(\sum_{i}\norm{dt_{i}(v)})^{2}
\]
Let $C$ be a constant such that 
$C^{2}\geq (\sum_{i}\norm{dt_{i}})^{2}$.
Then $\norm{ds_{\tau}}\leq C\epsilon$. Using lemma 
\ref{le:elementary}, we see that by making $\epsilon$ sufficiently small, 
we can ensure that the norm of the section $s$ 
is arbitrarly small. Finally, cover $M\cap rD^{d+n}$ by a 
a finite number of open sets, so that the tangent bundle of $M$
has a family of orthonormal sections on each of these open sets,
and repeat the argument for each of these open sets.  
\end{proof}

\section{Relation to $A$-theory}
\label{sec:Atheory}
This section describes a relation between the classifying spaces of
the embedded cobordism categories and Waldhausen's $A$-theory. More
precisely we shall describe a map
\[
\tau:\Omega B{\cal C}_{{d,n}}\to A(G(d,n))
\]
where $A(G(d,n))$ is Waldhausen's $K$-theory of the Grassmannian
$G(d,n)$. The map is an infinite loop map if $n=\infty$.
\subsection{A convenient model for $A$-theory}
Recall first the standard definition of $A(X)$ from
\cite{Atheory}. Let $R_{hf}(X)$ be the category of homotopy finite
retractive spaces over $X$. We work in the category of compactly
generated Hausdorff spaces. It has objects $(Y,r,s)$ where $r:Y\to X$,
$s:X\to Y$ and $rs=\mathrm{id}_{X}$, such that $(Y,s(X))$ is homotopy
equivalent to a finite CW complex relative to $s(X)$. This is a
category with cofibrations and weak equivalences, that is a
Waldhausen category.  

A map $i:Y_{1}\to Y_{2}$ over $X$ is a cofibration if the underlying
map (forgetting $X$) is a cofibration. Since we work in the category
of Hausdorff spaces all cofibrations are closed cofibrations(\cite{DoldHalbexakt}).
It is a weak equivalence if $i_{*}:\pi_{k}(Y_{1})\to
\pi_{k}(Y_{2})$ is a bijection for all $k$. 

Let $S_{\bullet}(X)$ denote Waldhausen's $S_{\bullet}$ construction
applied to $R_{hf}(X)$. An element of $S_{q}(X)$ is a flag 
\[
X \xrightarrow{s} Y_{1} \rightarrowtail Y_{2} \dots \rightarrowtail
Y_{q} \xrightarrow{r} X,
\]
together with a choice of quotients $Y_{j}/Y_{i} :=
Y_{j}\cup_{Y_{i}}X$, such that $X \to Y_{i}\to X$ is in
$R_{hf}(X)$. The graded set $S_{\bullet}(X)$ is a simplicial set
($d_{0}$ divides out $Y_{1}$ while $d_{j}$ omits $Y_{j}$ when $j\geq
1$). Weak equivalences of flags define a simplicial category
$wS_{\bullet}(X)$ whose nerve is the bi simplicial set 
$N_{\bullet}^{w}S_{\bullet}(X)$, and
\begin{equation}
\label{eqn:Adef}
A(X) =\Omega \vert N_{\bullet}^{w}S_{\bullet}(X)\vert.  
\end{equation}
In order to compare the embedded cobordism category with $A$-theory
we need a variant of Waldhausen's construction which we now turn
to. Let $B$ be a locally compact CW-complex satisfying \ref{B2} below. 
Let
\begin{equation}
\label{eqn:Retractions}
W(X,B)\subset R_{hf}(X)
\end{equation}
be the subset of retractive space $(Y,r,s)$ over
$X\times B$ with the two extra requirements:
\begin{enumerate}[(B1)]
\item 
\label{B1}
$Y\xrightarrow{r} X\times B \xrightarrow{\mathrm{id}_{B}} B$
is a fibration,
\item 
\label{B2}
$Y\xrightarrow{\Delta} Y\times Y$ is a cofibration.
\end{enumerate}

We use the term fibration to mean a surjective Horowitz fiber
space. A Hausdorff space which
satisfies \ref{B2} is called locally unconnected(LEC), see
\cite{DyerEilenberg}, \cite{Lewis:suspension} for a discussion of this category.
We also refer the reader to \cite{MaySigurdsson}, in particular chapter 4.

\begin{lemma}
For $X$ and $B$ LEC, $W(X,B)$ is a Waldhausen subcategory of
$R_{hf}(X\times B)$, provided $B$ is a locally compact $CW$ complex.
\begin{proof}
We must verify the axioms of \cite{Atheory}. Only the cofibration axioms  
needs to be checked. The initial object is  $X\times B$ which is
LEC, since both $X$ and $B$ are. For $(Y,r,s)\in W(X,B)$ the map
$X\times B \xrightarrow{s} Y$ is the inclusion of a retract of an LEC
and hence a cofibration \cite{DyerEilenberg}(theorem II.7), \cite{Lewis:suspension}(lemma 2.17). For cobase change: Given
\[
(Y_2,r_2,s_2)\xleftarrow{f}(Y_0,r_0,s_0)\overset i\rightarrowtail(Y_1,r_1,s_1)
\]
in $W(X,B)$ we must check that the adjunction space
$(Y_2\cup_{f}Y_1,r_2\cup_{f}r_1,s_2\cup_{f}s_1)$ is in $W(X,B)$. The
total space $Y_2\cup_{f}Y_1$ is LEC by the adjunction theorems of
\cite{DyerEilenberg} or \cite{Lewis:suspension}(theorem 2.3). Finally
\[
r_2\cup_{f}r_1:Y_2\cup_{f}Y_1 \to B
\]
is a fibration by \cite{Arnold}(theorem 2.5).
\end{proof}
\end{lemma}
We define $A(X,B)$ to be the $K$-theory of $W(X,B)$, 
\begin{equation}
  \label{eq:AlgKSubcategory}
A(X,B):=\Omega \vert N^{w}_{\bullet}S_{\bullet}(W(X,B))\vert .  
\end{equation}
It is equal to $A(X)$ when $B$ is a one-point space. We next examine
$A(X,B)$ for fixed $X$ and varying $B$.

For a map $f:B_{1}\to B_{2}$ we shall construct a contravariant
functor
\begin{equation}
  \label{eq:Contravariant}
f^{*}:W(X,B_{2})\to W(X,B_{1})  
\end{equation}
and, provided that $f$ is a fibration, also a covariant functor
\begin{equation}
  \label{eq:Covariant}
f_{*}:W(X,B_{1})\to W(X,B_{2}).
\end{equation}
The functors are ``exact'' in the sense that they preserve
cofibrations and weak equivalences - they are functors of Waldhausen
categories.

The contravariant functor is defined via pullback under
$\mathrm{id} \times f:X\times B_1\to X\times B_2$. Given 
$(Y_2,r_2,s_2)\in W(X,B_2)$, let
\[
Y_{1}=\{(y_2,x,b_1)\vert (x,f(b_1))=r_{2}(y_{2})\} 
\]
This is the total space of the pull-back by $f$ of the fibration
$Y_2\to B_2$, so $Y_1$ is LEC, \cite{Heath}. There are 
obvious maps 
$r_{1}:Y_{1}\to X\times B_{1}$ and
$s_{1}:X\times B_{1}\to Y_{1}$ defining an element of $W(X,B)$. 

\begin{lemma}
The pull-back $f^*:W(X,B_2)\to W(X,B_1)$ is a functor of Waldhausen
categories.
\begin{proof}
Let $(Y_2,r_2,s_2) \rightarrowtail (Y_2',r_2',s_2')$ be a cofibration in
$W(X,B)$. Since $Y_2$ and $Y_2'$ fibers over $B$, it follows from
\cite{HeathKamps} that $i:Y_2\to Y_2'$ is a cofibration over $B$,
that is, there is a fibrewise retraction 
\[
\xymatrix{
Y_2'\times I\ar[dr]\ar[rr]^{\pi_2}&&Y_2'\times \{0\}\cup Y_2\times I\ar[dl]\\
&B_2\ .&\\
}
\]
of the obvious inclusion. Let 
$(Y_1,r_1,s_1)=f^*(Y_2,r_2,s_2)$ and
$(Y_1',r_1',s_1')=f^*(Y_2',r_2',s_2')$. Then
\[
\pi_1(y_2',t,b_1):=(\pi_2(y_2',t),b_12)
\]
defines a retraction $Y_1'\times I \to Y_1'\times \{0\}\cup Y_1\times I$
Hence $f^*$ preserves cofibrations.

The functor $f^*$ also preserves weak equivalences since it maps fibrations to
fibrations. 
\end{proof}
\end{lemma}

The covariant structure $f_{*}$ is induced from  
\[
f_{*}:=(X\times f)_{*}:R_{hf}(X\times B_{1})\to R_{hf}(X\times B_{2})
\]
that sends $(Y_{1},r_{1},s_{1})$ to $(Y_{2},r_{2},s_{2})$ with
\[
Y_{2}=X\times B_{2}\cup_{X\times f} Y_{1}
\]
We must show that it defines an element of $W(X,B_{2})$.
This follows from the references and arguments above. We
remark that  $f:B_1\to B_2$ being a fibration implies that $Y_1\to B_1\to B_2$ 
is a fibration so that \cite{Arnold} applies to show that $Y_2\to B_2$ is
a fibration.

% Since $Y_{1}$ is LEC, $s_{1}:X\times B_{1}\to Y_{1}$ is a cofibration
% by \ref{LEC1}, and $X\times B_{1}$ is LEC.  One more application of \ref{LEC1} implies that $X\times B_1$ is LEC. Since $B_2$ is assumed LEC, $X\times B_2$ is LEC by \ref{LEC2}. The Dyer-Eilenberg adjunction theorem \ref{LEC3} shows that $y_2$ is LEC.

% To prove that $Y_2\xrightarrow{r_2} X\times B_2$ is a fibration, we can appeal to \ref{LEC5} applied to the diagram
% \[
% \xymatrix{
% X\times B_2\ar[dr] & X\times B_1 \ar[d]\ar@{>->}[r]\ar@{>->}[l]& Y_1\ar[dl] \\
% &B_2&
% }, p=f \circ \mathrm{pr}_{B_1}\circ r_1.
% \] 
% Thus $(Y_2,r_2,s_2)\in W(X,B_2)$.

\begin{theorem}
\label{th:homotopicMaps}
Homotopic maps $f,g:B_1\to B_2$ induce homotopic maps
\[
f^* \simeq g^* :
A(X,B_2)\to
A(X,B_1)
\]  
\begin{proof}
  We will show that the inclusion $i_0:B\times \{0\} \to B\times I$ 
induces a homotopy equivalence
\[
i_0^*:\vert N^w_\bullet S_\bullet X(X,B\times I)\vert \to
\vert N^w_\bullet S_\bullet X(X,B)\vert
\]
This suffices since the projection $B\times I \to B$ will also induce a homotopy equivalence,  and one can compose with the homotopy $B_1\times I \to B_2$ to complete the proof.

Let $h:I\to I$ be the constant map at 0. We must prove that it induces a homotopy equivalence
$h^*$ of $\vert N^w_\bullet S_\bullet X(X,B\times I)\vert$. 

This follows if we can show that for each $n$, the functor
\[
wS_n(W(X,B\times I)) \to wS_n(W(X,B\times I))
\]
induced by $h^*$, is connected to the identity by a sequence of natural transformation. Here $wS_n(-)$ denotes the category with objects $S_n(-)$ and weak equivalences as morphisms. To this end consider
\begin{gather*}
\mu: I \times I \to I, \mu(s,t)=st \\
\pi: I \times I \to I, \pi(s,t)=s,
\end{gather*}
and the canonical functor
\[
H^* : W(X,B\times I) \xrightarrow{\mu^*}
W(X,B\times I \times I) \xrightarrow{\pi_*} W(X,B\times I)
\]

Let $i_0,i_1:I \to I\times I$ be the maps $i_\nu(s)=(s,\nu)$. There is an induced diagram
\[
\xymatrix{
&W(X,B\times I)\ar[dr]^{\mathrm{id}}&\\
H^*:W(X,B\times I)\ar[ur]^{\mathrm{id}}\ar[r]^{\mu^*}\ar[dr]^{h^*}
&W(X,B\times I\times I) \ar[u]^{i_1^*}\ar[d]_{i_0^*}\ar[r]^{\pi_*}
&W(X,B\times I)\\
&W(X,B\times I)\ar[ru]^{\mathrm{id}}&
}
\]
The three functors from $W(X,B\times I)$ to itself are connected by natural transformations
\[
\mathrm{Id} \to H^* \gets h^*
\]
in the category $wW(X,B\times I)$ with morphisms being weak 
homotopy equivalences, cf. remark \ref{rem:whe} below. 
There is an induced diagram with $W(X,B\times I)$ replaced by 
$wS_nW(X,B\times I)$ and induced natural transformations. Consequently:
\begin{equation}
h^*\simeq id : N^w_\bullet S_n(W(X,B\times I)   
\end{equation}
for all $n$, so that 
\[
h^*\simeq \mathrm{id}:
A(X,B\times I)\to A(X,B\times I)
\]
by standard simplicial techniques.
\end{proof}
\end{theorem}
\begin{remark}
\label{rem:whe}
Given $f:B\to C$, $i:C\to B$ with $f\circ i=\mathrm{id}$. Let
$(Y,r,s)\in W(X,B)$ and $Y_0=i^*(Y,r,s)$. Then the inclusion
\[
Y_0 \to X\times C\cup_ {1\times f} Y
\]
is a weak homotopy equivalence if $f:B\to C$ is a weak homotopy equivalence.  
\end{remark}
We now let $B$ vary over the standard simplices $\Delta^p$ to get a simplicial space
\begin{equation}
\label{eq:parametrizedA}
  [p] \mapsto \vert N^w_\bullet S_\bullet (X, \Delta^p)\vert,
\end{equation}
where the simplicial maps are induced from the 
standard face and degeneracy maps $\Delta^p \to \Delta^q$
via the contravariant structure 
(\ref{eq:Contravariant}). 

It follows from Theorem (\ref{th:homotopicMaps}) that all structure maps in 
(\ref{eq:parametrizedA}) are weak homotopy equivalences. Thus
\begin{corollary}
\label{cor:WX}
  \[
\vert N^w_\bullet S_\bullet (W, \Delta^\bullet)\vert \simeq
\vert N^w_\bullet S_\bullet (X)\vert.
\]
\end{corollary}

\subsection{The map to $A$-theory}

Recall from section \ref{sec:homotopy-type} the two versions $N_\bullet^\delta {\cal C}_{d,n}^k$
and $N_\bullet {\cal C}_{d,n}^k$ with weakly equivalent geometric realizations. 
In this section we define a simplicial map 
\begin{equation}
  \label{eq:MapToA}
  \tau:\mathrm{sin}_\bullet N_\bullet^\delta {\cal C}_{d,n}^k \to
S^{(k)}_{\bullet} W(D(d,n),\Delta),
\end{equation}
where $\mathrm{sin}_\bullet(X)$ denotes the simplicial set which in degree
$p$ consists of singular simplices $\Delta^p\to X$, $S_\bullet^{(k)}$ is the $k$-fold iterated $S_\bullet$- construction and $W(G(d,n),\Delta^p)$ the Waldhausen category defined above.

We start with the case $k=1$, where we write ${\cal C}_{d,n}$ instead of ${\cal C}_{d,n}^1$. We must define 
\[
\tau_{p,q}:\mathrm{sin}_p N^\delta_ {d,n}\to
S_qW(G(d,n),\Delta^p)
\]
compatible with the bisimplicial structure maps. This requires some preparations about the structure of $N_q^\delta {\cal C}_{d,n}$ which we now turn to. See also \S 2.1 of \cite{GMTW}. 

Let $W^d$ be an abstract (as opposed to embedded) cobordism from $M_0$ to $M_1$, equipped with disjoint collars
\[
h_0:[0,1]\times M_0 \to W,\quad
h_1:[0,1]\times M_1 \to W
\]
and let $\mathrm{Emb}_\epsilon(W,[0,1]\times \RR^{d+n-1})$ denote the space of smooth embeddings
\[
e: W \to [0,1]\times \RR^{d+n-1}
\]
such that there are embeddings $e_\nu: M_\nu \to \RR^{d+n-1}$ with
\[
e\circ h_0(t_0,x_0)=(t_0,e_0(x_0)),\quad
e\circ h_1(t_1,x_1)=(t_1,e_1(x_1))
\]
where $e_0\in [0,\epsilon)$ and $e_1\in (1-\epsilon,1]$. Similarly, let
$\mathrm{Diff}_\epsilon(X)$ be the group of diffeomorphisms that restricts to product diffeomorphisms on the $\epsilon$-collars. We let 
$\mathrm{Emb}(-,-)$ and $\mathrm{Diff}(-,-)$  denote the colimits as
$\epsilon \to 0$, Define
\begin{align*}
  E_n(W)&:=\mathrm{Emb}(W,[0,1]\times \RR^{d+n-1})\times_{\mathrm {Diff}(W)}(W)\\
  B_n(W)&:=\mathrm{Emb}(W,[0,1]\times \RR^{d+n-1})/{\mathrm{Diff}(W)}
\end{align*}
The projection $\pi:E_n(W)\to B_n(W)$ is a smooth fiber bundle of infinite dimensional smooth manifolds in the convenient topology of \cite{KrieglMichor}, in fact an embedded bundle in the sense of the diagram
\[
\xymatrix{
E_n(W)\ar[d]^\pi\ar@{^{(}->}[r]& B_n(W) \times \RR^{d+n}\ar[dl]\\
B_n(W)&
}
\]
Moreover, a smooth map $B^m \to B_n(W)$ from a finite dimensional  manifold $B^m$ induces smooth embedded fiber bundle of finite dimensional manifolds
\[
\xymatrix{
E^{m+d}\ar[d]^\pi\ar@{^{(}->}[r]& B^m \times \RR^{d+n}\ar[dl]\\
B^m&\ ,
}
\]
and continuous maps into $B_n(W)$ can be approximated by smooth maps, so the set of homotopy classes of continuous maps from $B^m$ to $B_n(W)$ is equal to the set of homotopy classes of continuous maps.

For a closed $(d-1)$- dimensional manifold $M^{d-1}$ there is a similar smooth fiber bundle of infinite dimensional manifolds
\[
\xymatrix{
E_n(M)\ar[d]^\pi\ar@{^{(}->}[r]& B_n(M) \times \RR^{d+n-1}\ar[dl]\\
B_n(M)&
}
\]
as in \S 2.1 of \cite{GMTW}.

We topologize $N_1^\delta {\cal C}_{d,n}$ as the disjoint union of the object space $N_0^\delta{\cal C}_{d,n}$ and of the space of non-identity morphisms. Then there is a homeomorphism
\begin{equation}
  \label{eq:N1homeo}
N_1^\delta {\cal C}_{d,n}\cong
\coprod_{\{M\}}(B_n(M^{d-1})\times \RR^\delta)\sqcup
\coprod_{\{W\}}(B_n(W^{d})\times (\RR^2_+)^\delta)
\end{equation}
where the disjoint union is over certain diffeomorphism classes of 
closed $(d-1)$-manifolds, respectively compact $d$-dimensional cobordisms, 
namely the diffeomorphism classes that embed in 
$\RR^{d+n-1}$ resp. $\RR^{n+d}$. 
We note that (\ref{eq:N1homeo}) gives 
$N_1^\delta{\cal C}_{d,n}$ the structure of an infinite dimensional smooth manifold.

Let $\sigma :\Delta^p \to N_1^\delta{\cal C}_{d,n}$ be a smooth $p$-simplex,
landing in a non-identity component. This induces a smooth embedded 
fiber bundle.
\[
\xymatrix{
E[a_0,a_1] \ar[d]^\pi\ar@{^{(}->}[r]& \Delta^p\times [a_0,a_1]\times \RR^{d+n-1}\ar[dl]\\
\Delta^p&, \mathrm{dim} E[a_0,a_1]=p+d.
}
\] 
For $z\in E[a_0,a_1]$, the vertical tangent space $T_z^\pi E[a_0,a_1]$ is a subspaces of $\{\pi(z)\}\times \RR^{d+n}$. This defines a map
\[
\tau: E[a_0,a_1]\to G(d,n)
\]
into the Grassmannian. Let $E(a_0)$ be the left-hand boundary of $E[a_0,a-1]$ and consider the retractive space
\begin{equation}
  \label{eq:RetractiveSpace}
  G(d,n)\times \Delta^p\xrightarrow{s}
E[a_0,a_1]\cup_{E(a_0)}G(d,n)\times \Delta^p\xrightarrow{r}
G(d,n)\times \Delta^p
\end{equation}
with $r=(\tau,\pi)$ on $E[a_0,a_1]$ and the identity on 
$G(d,n)\times \Delta^p$. The composition
\begin{equation*}
E[a_0,a_1]\cup_{E(a_0)}G(d,n)\times \Delta^p\xrightarrow{r}
G(d,n)\times \Delta^p\to \Delta^p
\end{equation*}
is a fibration with LEC total space. Thus (\ref{eq:RetractiveSpace}) defines an element of the Waldhausen category $W(G(d,n),\Delta^p)$. The resulting map
\[
\mathrm{sin}_p(N_1^\delta{\cal C}_{d,n})
\to
W(G(d,n),\Delta^p)
\]
respects the simplicial identities as $p$ varies. Quite similarly, a singular $p$-simplex of 
$N_q^\delta{\cal C}_{d,n}$ defines a sequence of codimension zero embeddings
\[
E[a_0,a_1]\subset E[a_0,a_2]\subset\dots \subset
E[a_0,a_q]\subset \Delta^p\times [a_0,a_q]\times \RR^{d+n}
\]
fibering over $\Delta^p$ with 
$E[a_0,a_{i+1}]=E[a_0,a_i]\cup_{E(a_i)} E[a_i,a_{i+1}]$.
This amounts to a map 
\[
\tau_{p,q}:\mathrm{sin}_pN^\delta_q{\cal C}_{d,n}\to
S_qW(G(d,n), \Delta^p)
\]
that gives rise to a bisimplicial map
\[
\tau_{\bullet,\bullet}:\mathrm{sin}_\bullet N^\delta_\bullet{\cal C}_{d,n}\to
S_\bullet W(G(d,n),\Delta^p).
\]
We can include $S_\bullet(W)$ into $N_\bullet^wS_\bullet(W)$
and get by corollary \ref{cor:WX}:
\[
\tau: \vert \mathrm{sin}_\bullet N^\delta_\bullet{\cal C}_{d,n}\vert\to
\vert N^w_\bullet S_\bullet(G(d,n)\vert.
\]
Finally, the canonical map $\vert \mathrm{sin}_\bullet(X)\vert \to X$
is a weak equivalence, and $\vert N^\delta _\bullet{\cal C}_{d,n}\vert
\sim \vert N_\bullet{\cal C}_{d,n}\vert$ by \ref{sec:homotopy-type}.
This proves
\begin{theorem}
  Tangents along the fiber induces a weak map
\footnote{A weak map from $X$ to $Y$ is a composite of the form $X \gets X^\prime \to Y$ with $X^\prime \to X$ a weak homotopy equivalence, for instance an invertible map in the homotopy category associated to a model category defining the weak homotopy equivalences.}
\[
\Omega B{\cal C}_{d,n}\to A(G(d,n)).
\tag*{$\square$}
\]
\end{theorem}
The remainder of this section will argue that the map $\tau$ in the above theorem is an infinite loop map (when $n=\infty$). Theorems \ref{th:delooping} and \ref{prop:he} imply that the classifying space $B{\cal C}^k_ {d,n}$ is a $(k-1)$-fold deloop of $B{\cal C}_{d,n}$, provided that $k\leq d+n$, For a Waldhausen category ${\cal C}$, the iterated $S_\bullet$-construction $S_\bullet^{(k)}$ deloops $S_\bullet{\cal C}$ by proposition 1.5.3 of \cite{Atheory}.  Since 
$N_\bullet^\delta{\cal C}^k_{d,n}$ is weakly equivalent to 
$N_\bullet{\cal C}^k_{d,n}$, the delooping of $\tau$ is achieved by extending its definition to a multi-simplicial map 
\begin{equation}
  \label{eq:Deloop}
  \tau^k:\mathrm{sin}_\bullet N_\bullet^\delta {\cal C}^k_{d,n}\to
S_\bullet^{(k)}W(G(d,n),\Delta^\bullet)
\end{equation}
for $k\leq d+n$.

The construction of $\tau^k$ is completely similar to the case of $k=1$; we give the details for $k=2$. Let $a_0<a_1<\dots <a_p$ and 
$b_0<b_1<\dots <b_q$ be two sequences of real numbers. Write
\begin{align*}
  J_{i,j}&=[a_0,a_i]\times [b_0,b_j]\subset \RR^2\\
\partial_0 J_{i,j} &= \{a_0\}\times [b_0,b_j]\cup [a_0,a_i]\times \{b_0\}.\\
\end{align*}
Given a smooth singular simplex $\sigma :\Delta^s\to N^\delta_{p,q}{\cal C}^2_{d,n}$
we get sequences 
 $\underline a =(a_0<a_1<\dots <a_p)$, 
$\underline b=(b_0<b_1<\dots <b_q)$ that do not vary with $z\in \Delta^s$ and
\[
E^{p+q}\subset \Delta^s\times J_{p,q}\times \mathrm{int}(I^{d+n-2}). 
\]
Let $E_{i,j}=E\cap \Delta^s\times J_{i,j} \times \mathrm{int}(I^{d+n-2})$.
It is a compact manifold with corners and the projection
$E_{i,j} \to \Delta^s$ is a smooth fiber bundle where tangents along the fibers
give compatible maps from $E_{i,j}$ to $G(d,n)$. Form
\[
Y_{i,j} = E_{i,j}\cup_{\partial_0 E_{i,j}} G(d,n)\times \Delta^s\in
W(G(d,n), \Delta^s)
\]
where $\partial_0 E_{i,j}=E_{i,j}\cap \partial_0 J_{i,j}$. The diagram
\[
\xymatrix{
Y_{1,1}\ar@{>->}[r]\ar@{>->}[d]
&Y_{2,1}\ar@{>->}[r]\ar@{>->}[d]& \dots\ar@{>->}[r] &Y_{p,1}\ar@{>->}[d]\\
Y_{1,1}\ar@{>->}[r]\ar@{>->}[d]
&Y_{2,1}\ar@{>->}[r]\ar@{>->}[d]& \dots\ar@{>->}[r] &Y_{p,1}\ar@{>->}[d]\\
\vdots\ar@{>->}[d]&\vdots\ar@{>->}[d]&&\vdots\ar@{>->}[d]\\
Y_{1,1}\ar@{>->}[r]
&Y_{2,1}\ar@{>->}[r]& \dots\ar@{>->}[r] &Y_{p,1}\\
}
\]
represents an element of $S_\bullet S_\bullet W(G(n,d), \Delta^s)$,
and the resulting map
\[
\mathrm{sin}_s N^\delta _{p,q} {\cal C}^2_{d,n}\to
S_\bullet S_\bullet W(G(n,d), \Delta^s)
\]
commutes with the simplicial structure maps . This defines the map
(\ref{eq:Deloop}) for $k=2$. The general case $k<2$ is entirely similar.

% \section{Appendix}
% \newcommand{\two}{{II}}
% Let $M\subset \RR^n$ be an embedded manifold. Assume that the second fundamental form of $M$ is bounded by $C>0$, so that for any $v \in \nu_pM$ we have that
% \[
% \norm{\two_p(v)}< < C {\norm{v}}^2 
% \] 

\bibliographystyle{plain}
\bibliography{regular}

\end{document}